\documentclass[11pt, a4paper]{article}
\usepackage{amssymb, amsmath}
\usepackage[utf8]{inputenc}
\usepackage[T1]{fontenc}
\usepackage[english]{babel}

\usepackage[height=22cm,width=15cm]{geometry}

\usepackage[usenames,dvipsnames]{color}

\usepackage[usenames,dvipsnames]{pstricks}
\usepackage{epsfig}
\usepackage{pst-grad} 
\usepackage{pst-plot} 

\usepackage{xfrac}

\usepackage{euscript}
\usepackage{amsmath}
\usepackage{amsfonts}
\usepackage{amssymb}
\usepackage{amsthm}
\usepackage{float}
\usepackage{graphicx}

\usepackage[hidelinks,pdfusetitle]{hyperref}
\usepackage{xcolor}
\usepackage[capitalise]{cleveref}

\newcommand{\N}{\mathbb{N}}
\newcommand{\R}{\mathbb{R}}

\newtheorem{theorem}{Theorem}
\newtheorem{proposition}{Proposition}[section]
\newtheorem{lemma}[proposition]{Lemma}
\newtheorem{remark}[proposition]{Remark}

\numberwithin{equation}{section}

\newcommand{\rev}[1]{{\color{black}#1}}

\newcommand{\rphi}{\mathfrak{r}}
\newcommand{\jrun}{r_1}
\newcommand{\jrdeux}{r_2}
\newcommand{\run}{\rphi_1}
\newcommand{\rdeux}{\rphi_2}

\newcommand{\opphi}{e^{\rho |\partial_x|}\, }
\newcommand{\oprhozero}{e^{\rho_0 |\partial_x|}}
\newcommand{\opphibeta}{e^{\rho |\partial_x|-\beta}}

\newcommand{\ldiv}{L^2_{\mathrm{div}}}
\newcommand{\lloc}{L^2_{\mathrm{loc,div}}}

\newcommand{\philp}{\varphi_{\mathrm{lp}}}
\newcommand{\chilp}{\chi_{\mathrm{lp}}}
\newcommand{\TT}{\mathbf{T}}
\newcommand{\RR}{\mathbf{R}}

\newcommand{\dd}{\,\mathrm{d}}
\newcommand{\dif}{\mathrm{d}}
\newcommand{\ds}{\dif s}
\newcommand{\dt}{\dif t}

\newcommand{\dz}{\dif z}

\newcommand{\diverg}{\mathop{\mathrm{div}}}
\newcommand{\supp}{\mathop{\rm supp}\nolimits}

\newcommand{\eval}[1]{ \{ #1 \} }
\newcommand{\ue}{u^\varepsilon}
\newcommand{\uapp}{u^\varepsilon_{\mathrm{app}}}
\newcommand{\uinit}{u_{*}}

\newcommand{\ex}{e_x}
\newcommand{\ey}{e_y}
\newcommand{\PN}{\mathbf{P}_N}

\newcommand{\newcom}{\newcommand}

\def\inte#1{
\displaystyle\mathop{#1\kern0pt}^\circ }

\newcom{\cC}{{\mathcal C}}
\newcom{\cD}{{\mathcal D}}
\newcom{\cF}{{\mathcal F}}
\newcom{\cL}{{\mathcal L}}
\newcom{\cM}{{\mathcal M}}
\newcom{\cP}{{\mathcal P}}
\newcom{\cS}{{\mathcal S}}
\newcom{\cQ}{{\mathcal Q}}
\newcom{\cT}{{\mathcal T}}
\newcom{\cY}{{\mathcal Y}}
\newcom{\cZ}{{\mathcal Z}}
\newcom{\T}{\mathbb{T}}
\newcom{\Z}{\mathbb{Z}}
\newcom{\E}{\mathbb{E}}

\let\e=\varepsilon
\newcom{\f}{\frac}
\newcom{\dint}{\displaystyle\int}
\newcom{\dsum}{\displaystyle\sum}
\newcom{\dlim}{\displaystyle\lim}
\newcom{\ov}{\overline}
\newcom{\wt}{\widetilde}
\newcom{\pa}{\partial}
\newcom{\p}{\partial}
\newcom\na{\nabla}
\newcom\rto{\rightarrow}
\newcom\lto{\leftarrow}
\newcom\mto{\mapsto}
\newcom{\disp}{\displaystyle}
\newcom{\non}{\nonumber}
\newcom{\no}{\noindent}
\newcom{\QED}{$\square$}

\newcommand{\D}{\dot{\Delta}}
\newcommand{\Low}{\dot{S}}

\newcommand{\band}{\mathfrak{B}}

\newcommand{\besov}[1]{\dot{B}^{#1}}
\newcommand{\cb}{C_{\mathrm{B}}}

\newcommand{\email}[1]{\href{mailto:#1}{\texttt{#1}}}
    
\title{Controllability of the Navier-Stokes equation in a rectangle with a little help of a distributed phantom force}

\author{Jean-Michel Coron\footnote{Laboratoire Jacques-Louis Lions,
    Universit\'e Pierre et Marie Curie,
    4 place Jussieu, 
    75252 Paris cedex 05; France, \email{coron@ann.jussieu.fr}} \footnote{ETH Zurich
Institute for Theoretical Studies
Clausiusstrasse 47
8092 Zurich; 
Switzerland}, 
    Frédéric Marbach\footnote{Univ Rennes, CNRS, IRMAR - UMR 6625, F-35000 Rennes, France, \email{marbach@ann.jussieu.fr}}, Franck Sueur\footnote{Institut de Math\'ematiques de Bordeaux,
    Universit\'e de Bordeaux,
    351 cours de la lib\'eration,
    33405 Talence; France, \email{franck.sueur@math.u-bordeaux.fr}}, Ping Zhang\footnote{Academy of Mathematics $\&$ Systems Science
and  Hua Loo-Keng Key Laboratory of Mathematics, The Chinese Academy of
Sciences, Beijing 100190; China, \email{zp@amss.ac.cn}}}

\usepackage{soul}

\begin{document}

\maketitle

\begin{abstract}

We consider the 2D incompressible Navier-Stokes equation in a rectangle with the usual no-slip boundary condition prescribed on the upper and lower boundaries. We prove that for any positive time, for any finite energy initial data, there exist controls on the left and right boundaries and a distributed force, which can be chosen arbitrarily small in any Sobolev norm in space, such that the corresponding solution  is at rest at the given final time.  

 Our work improves earlier results in~\cite{MR2269867, MR2994698} where the distributed force is small only in a negative Sobolev space. It is a further step towards an answer to Jacques-Louis Lions' question in~\cite{MR1147191} about the small-time global exact boundary controllability of the Navier-Stokes equation with the no-slip boundary condition, for which no distributed force is allowed.

Our analysis relies on the well-prepared dissipation method already used in~\cite{MR3227326} for Burgers and in~\cite{2016arXiv161208087C} for Navier-Stokes in the case of the Navier slip-with-friction boundary condition. In order to handle the larger boundary layers associated with the no-slip boundary condition, we perform a preliminary regularization into analytic functions with arbitrarily large analytic radius and prove a long-time nonlinear Cauchy-Kovalevskaya estimate relying only on horizontal analyticity, in the spirit of~\cite{MR2145938,MR3464051}.

\end{abstract}

\newpage

\setcounter{tocdepth}{2}
\tableofcontents

\newpage

\section{Introduction and statement of the main result}

\subsection{Historical context}

In the late 1980's, Jacques-Louis Lions introduced in~\cite{MR1147191} (see also \cite{MR1453828,everything,MR2051507}) the question of the controllability of fluid flows in the sense of how the Navier-Stokes system can be driven by a control of the flow on a part of the boundary to a wished plausible state, say a vanishing velocity. Jacques-Louis Lions' problem has been solved in~\cite{2016arXiv161208087C} by the first three authors in the particular case of the Navier slip-with-friction boundary condition (see also \cite{2017arXiv170307265C} for a gentle introduction to this result). In its original statement with the no-slip Dirichlet boundary condition, it is still an important open problem in fluid controllability.

\subsection{Statement of the main result}

In this paper we consider the case where the flow occupies a rectangle, where controls are applied to the lateral boundaries and the no-slip condition is prescribed on the upper and lower boundaries. We thus consider a rectangular domain 
$$\Omega := (0,L) \times (-1,1),$$
where $L > 0$ is the length of the domain. We will use $ (x,y)$ as coordinates. Inside this domain, a fluid evolves under the Navier-Stokes equation. We will name $u = (u_1,u_2)$ the two components of its velocity. Hence, $u$ satisfies:
\begin{equation} \label{eq.ns}
 \left\{
  \begin{aligned}
   \partial_t u + \left( u \cdot \nabla \right) u + \nabla p - \Delta u & = f_g, \\
   \diverg u & = 0,
  \end{aligned}
 \right.
\end{equation}
in $\Omega$, where $p$ denotes the fluid pressure and $f_g$ a force term (to be detailed below). We think of this domain as a river or a tube and we assume that we are able to act on the fluid flow at both end boundaries: 
$$\Gamma_0 := \{0\} \times (-1,1) \text{ and } \Gamma_L := \{L\} \times (-1,1).$$
On the remaining parts of the boundary, 
$$\Gamma_{\pm} := (0,L) \times \{\pm 1\},$$
we  assume that we cannot control the fluid flow and that it satisfies null Dirichlet boundary conditions: 
\begin{equation} \label{cd.ns}
  u = 0 
  \quad \text{on} \quad
  \Gamma_{\pm}.
\end{equation}
We will consider initial data in the space $\ldiv(\Omega)$ of divergence free vector fields, tangent to the boundaries $\Gamma_{\pm}$. The main result of this paper is the following. 

\begin{theorem} \label{thm.main1}
 Let $T> 0$ and $u_*$ in $\ldiv(\Omega)$. 
 For any $k \in \N$ and for any $\eta > 0$, there exists a force {$f_g \in L^1((0,T);H^k(\Omega))$} satisfying
 \begin{equation}
  \label{f.small}
  \left\| f_g \right\|_{L^1((0,T);H^k(\Omega))} \leq \eta
 \end{equation}
 and an associated weak Leray solution $u \in C^0([0,T];\ldiv(\Omega)) \cap L^2 ((0,T); H^1(\Omega))$ to~\eqref{eq.ns} and \eqref{cd.ns} satisfying $u(0) = u_*$ and $u(T) = 0$.
\end{theorem}

Since the notion of weak Leray solution is classically defined in the case where the null Dirichlet boundary condition is prescribed on the whole boundary, let us detail that we say that $u \in C^0([0,T];\ldiv(\Omega)) \cap L^2 ((0,T); H^1(\Omega))$ is a weak Leray solution to~\eqref{eq.ns} and~\eqref{cd.ns} satisfying $u(0, \cdot) = u_*$ and $u(T,\cdot) = 0$ when it satisfies the weak formulation
\begin{equation}
 \label{int}
 \begin{split}
 - \int_0^T\int_\Omega u \cdot \partial_t \varphi
 + \int_0^T\int_\Omega (u \cdot \nabla) u \cdot \varphi
 & + 2 \int_0^T\int_\Omega D(u) : D(\varphi)
 \\ & = \int_\Omega u_* \cdot \varphi(0,\cdot)
 +\int_0^T\int_\Omega u \cdot f_g  ,
 \end{split}
\end{equation}
for any test function $\varphi \in C^\infty([0,T]\times\bar{\Omega})$ which is divergence free, tangent to $\Gamma_{\pm}$, vanishes at $t = T$ and vanishes on the controlled parts of the boundary $\Gamma_0$ and $\Gamma_L$. Thus, \eqref{int} encodes the no-slip condition on the upper and lower boundaries only. This under-determination encodes that one has control over the remaining part of the boundary, that is on the lateral boundaries. The controls on the lateral boundaries are therefore not explicit in the statement of Theorem \ref{thm.main1}. Still the proof below will provide some more insights on the nature of possible controls to the interested reader. Once a trajectory is known, one can indeed deduce that the associated controls are the traces on $\Gamma_0$ and $\Gamma_L$ of the solution. \rev{Hence, elementary trace theorems and the regularity of weak Leray solutions imply that such controls are at least in $L^2((0,T);H^{\frac 12}(\Gamma_0 \cup \Gamma_L))$.}

\subsection{Comments and references}

\begin{remark}[Relation to the open problem]
 We view Theorem~\ref{thm.main1} as an intermediate step towards an answer to Jacques-Louis Lions' problem, which requires to prove that the theorem is still true with a vanishing distributed force $f_g = 0$. Here, we need a non-vanishing force but we can choose it very small even in strong topologies. Our result therefore suggests that the answer to Jacques-Louis Lions' question is very likely positive, at least for this geometry. Nevertheless, new ideas are probably necessary to ``eliminate'' the unwanted distributed force we use.
\end{remark}

\begin{remark}[Local vs.\ global null controllability and Reynolds numbers] \label{rk:local}
 The fact that, for any $T > 0$, one can drive to the null equilibrium state $u = 0$ in time $T$ without any distributed force ($f_g = 0$) was already known when the initial data $\uinit$ is small enough in $L^2(\Omega)$ (with a maximal size depending on $T$). In this case, one may think of the bilinear term in Navier-Stokes system
 as a small perturbation term of the Stokes equation so that the controllability can be obtained by means of Carleman estimates and fixed point theorems. Loosely speaking, such an approach corresponds to low Reynolds controllability. 
 
 More generally, local null controllability is a particular case of local controllability to trajectories. For Dirichlet boundary conditions, the first results have been obtained by Imanuvilov who proved local controllability in 2D and 3D provided that the initial state are close in $H^1$ norm, with interior controls, first towards steady-states in \cite{MR1617825} then towards strong trajectories in \cite{MR1804497}. Fursikov and Imanuvilov proved large time global null controllability in 2D for a control supported on the full boundary of the domain in \cite{MR1308746}. Still in 2D, they also proved local controllability to strong trajectories for a control acting on a part of the boundary and initial states close in $H^1$ norm in \cite{MR1422385}. Eventually, in \cite{MR1472291} they proved in 2D and 3D local controllability to strong trajectories with controls acting on the full boundary, still for initial states close in $H^1$ norm. More recently, these works have been improved in \cite{MR2103189}, where the authors proved local controllability towards less regular trajectories with interior controls and for initial states close in $L^2$ norm in 2D and $L^4$ norm in 3D.

 In contrast, in \cref{thm.main1}, the initial data $\uinit$ can be arbitrarily large (and $T$ arbitrarily small). This corresponds to controllability of Navier-Stokes system at large Reynolds numbers. In this regime, the first author and Fursikov proved global null controllability for the Navier-Stokes system in a 2D manifold without boundary in \cite{MR1470445}. 
\end{remark} 

\begin{remark}[Comparison with earlier results]
 For the large Reynolds regime, let us mention the earlier references~\cite{MR2269867, MR2994698} where a related result is obtained in a similar setting. In this earlier result, the distributed force can be chosen small in $L^p((0,T);H^{-1}(\Omega))$, where $1 < p < 4/3$.
 The fact that, in
 \cref{thm.main1}, our phantom force can be chosen arbitrarily small in the space $L^1((0,T),H^k(\Omega))$ for any $k \geq 0$, is the major improvement of this work.
\end{remark}

\begin{remark}[Geometric setting]
 Theorem~\ref{thm.main1} remains true for any rectangular domain $(0,L_1)\times(0,L_2)$ and any positive viscosity $\nu$, thanks to a straightforward change of variables. On the other hand, we consider the case of a rectangle because it provides many crucial simplifications. This is not for the sake of clarity of the exposition; we suspect that the case of a general domain requires different arguments. A key point is that this geometric setting and the use of well-chosen controls enable us to guarantee that the boundary layer equations we consider will remain linear and well-posed (see \eqref{eq.V} and \cref{remark-blc}).
 \rev{As we use a flushing strategy in the horizontal direction, we make use of controls on the two lateral sides, and on the two components of the velocity fields, on the contrary to 
\cite{CL,2006} where the case of a small initial data is considered.

Moreover there is a very good chance that the result of  Theorem~\ref{thm.main1} also holds in the case where the domain $\Omega$ is the  parallelepiped 
  $ (0,L_1) \times (0,L_2) \times (-1,1),$ where $L_1, L_2 > 0$, with only two opposite uncontrolled faces on which the no-slip condition is imposed. Yet another highly likely extension is the case where  the domain $\Omega$ is the circular  duct $ D\times (0,L)$, where $D$ is the unit  disk and $L>0$, with controls on the lateral sides $ D \times \{0\}$ and $D \times \{L\}$ and with the no-slip condition on $ T\times (0,L)$. 
 }
\end{remark}

\begin{remark}[Additional properties of the phantom force]
 During the proof, we will check that the phantom force $f_g$ we use has $C^\infty$ regularity by parts with respect to time and $C^\infty$ regularity with respect to space for each time.
 Moreover, we will check that, during the most important step of our strategy (the global approximate control phase, which involves passing through intermediate states of very large size), there exists $\delta > 0$ such that
 \begin{equation}
  \label{f.support}
  \supp f_g(t) \subset [0,L] \times [-1+\delta ,1-\delta ].
 \end{equation}
\rev{In the case where the initial data $u_*$ is smooth, the smoothness of the phantom force
guarantees the existence of a smooth controlled solution satisfying the conclusion of Theorem \ref{thm.main1}.  
Moreover, we think that this also holds true for the two particular 3D settings mentioned above, ruling out the possibility of a blow up in these cases. 
}
\end{remark}

\rev{
\begin{remark}[Comparison with results on wild solutions]
Let us highlight that the solution mentioned in Theorem \ref{thm.main1} is 
 not a « wild » solution. Indeed, let us recall that, in 2D, without control, weak Leray solutions are unique and regular for positive times. The controlled solution that is considered in Theorem \ref{thm.main1} benefits from the same regularity than the one in the classical Leray theory.  This solution  vanishes at some positive time thanks to the action of a well chosen regular source term. This has to be distinguished from the recent results  \cite{BCV,BV2,BV1} where  wild weak solutions (with a regularity less that the one considered by Leray) of the $3$D Navier-Stokes equations are constructed, and these solutions may vanish at some positive finite time too, without any control.
\end{remark}}
\rev{
Finally let us foreshadow that the proof of Theorem~\ref{thm.main1} will make great use of analyticity technics. Let us therefore refer to \cite{HS} for a glimpse of classical and recent analyticity results on the Navier-Stokes equations without control. }

\subsection{Strategy of the proof and plan of the paper}

We explain the strategy of the proof of \cref{thm.main1}, which is divided into three steps. First, we prove that the initial data can be regularized into an analytic function with arbitrary analyticity radius. Then, we prove that a large analytic initial data with a sufficient analyticity radius can be driven approximately to the null equilibrium. Last, we know that small enough states can be driven exactly to the rest state. These three steps are implemented in the three propositions below, where we set the Navier-Stokes equations in the horizontal band 
\begin{equation} \label{band}
 \band := \mathbb{R} \times [-1,1],
\end{equation}
with a control supported in the extended region, outside of $\Omega$. Therefore, we look for solutions to
\begin{equation} \label{eq:ns-fc-fg}
 \left\{
  \begin{aligned}
   \partial_t u + (u \cdot \nabla) u + \nabla p - \Delta u & = f_c + f_g
   && \text{in } (0,T) \times \band, \\
   \diverg u & = 0
   && \text{in } (0,T) \times \band, \\
   u & = 0
   && \text{on } (0,T) \times \partial\band,
  \end{aligned}
 \right.
\end{equation}
where the force $f_c$ is a control supported in $\band\setminus\Omega$ and the force $f_g$ is the phantom (ghost) force supported in $\Omega$. Restricting such solutions of Navier-Stokes in the band to the physical domain $\Omega$ will prove \cref{thm.main1}. We also introduce a domain 
\begin{equation} \label{def-G}
G := [-2L,-L]\times[-1,1] \quad \text{ (see \cref{fig:ghost}).} 
\end{equation}
We will denote by $\ex$ and $\ey$ the unit vectors of the  canonical basis of $\R^2$.

\begin{proposition}[Analytic regularization of the initial data] \label{prop_A}
  Let $T >0$ and $\rho_b > 0$. Let $\uinit \in \ldiv(\bar{\Omega})$ \rev{with $\uinit \cdot e_x = 0$ on $\Gamma_{0}$ and $\Gamma_L$}. For any $k\in \N$ and $\eta_b > 0$, there exists an extension $u_a \in \ldiv(\band)$ of $\uinit$ to the band $\band$, a control force $f_c \in C^\infty([0,T]\times(\band\setminus\bar{\Omega}))$, a phantom force $f_g \in C^\infty([0,T]\times \bar{\Omega})$ satisfying
 \begin{gather}
  \label{fa.small-rest}
  \| f_g \|_{L^1 ((0,T) ; H^k (\Omega))} \leq \eta_b, 
 \end{gather}
 a weak Leray solution {$u \in C^0([0,T];\ldiv(\band)) \cap L^2((0,T); H^1(\band))$} to \eqref{eq:ns-fc-fg} associated with the initial data $u_a$, $C_b >0$ and $T_b \leq T$ such that $u_b := u(T_b) \in \ldiv(\band)$  satisfies
 \begin{gather} \label{eq:ub-estimate}
  \| {u_b}_{\rvert G} \|_{H^k(G)} + 
  \| {u_b}_{\rvert\band\setminus\Omega} \|_{L^2(\band\setminus\Omega)} +
  \| {u_b}_{\rvert\Omega} - {\uinit} \|_{L^2(\Omega)} 
  \leq \eta_b,
  \\
  \label{ub.anal}
  \forall m \geq 0, \quad \| \partial_x^m u_b \|_{H^3(\band)} \leq \frac{m!}{\rho_b^m} C_b,
  \\
  \label{ub.int}
  \sum_{0 \leq \alpha + \beta \leq 3} \| \p_x^{\alpha} \p_y^{\beta} u_b \|_{L^1_x(L^2_y)} \leq C_b.
 \end{gather}
\end{proposition}

\begin{proposition}[Global approximate null controllability from any analytic initial data] \label{prop_B}
 Let $T > 0$. There exists $\rho_b > 0$ such that, for every $\sigma > 0$ and each $u_b \in \ldiv(\band)$ for which there exists $C_b > 0$ such that \eqref{ub.anal} and \eqref{ub.int} hold, for every $k\in\N$ and $\delta\in(0,\frac{1}{2})$, there exist two forces $f_c \in C^\infty([0,T]\times (\band \setminus \bar \Omega))$ and $f_g \in C^\infty([0,T]\times \bar{\Omega})$ satisfying
 \begin{gather}
  \label{f.small-rest}
  \| f_g \|_{L^1 ((0,T) ; H^k (\Omega))} \leq C_{k,\delta} \| {u_b}_{\rvert G} \|_{H^k(G)} , \\
  \label{f.support-rest}
  \supp f_g  \subset (0,T) \times [0,L] \times [-1+\delta ,1-\delta ],
  \\
  \label{fc.compactB}
  \supp f_c \subset (0,T) \times \band \setminus \bar{\Omega}
 \end{gather}
 and a weak solution $u \in C^0([0,T];\lloc(\band)) \cap L^2((0,T); H^1_{\mathrm{loc}}(\band))$ to \eqref{eq:ns-fc-fg} associated with the initial data $u_b$, such that there exists $T_c \leq T$ such that $u_c := u(T_c) \in \lloc(\band)$ satisfies
 \begin{equation} \label{eq:uc-estimate}
  \| {u_c}_{\rvert \Omega} \|_{L^2(\Omega)} \leq \sigma + \| {u_b}_{\rvert \{ |y| \geq 1-2\delta\}} \|_{L^2(\band)}.
 \end{equation} 
 Moreover, the constant $C_{k,\delta}$ only depends on $k$ and $\delta$.
\end{proposition}

\begin{proposition}[Local null controllability] \label{prop_C}
 Let $T > 0$. There exists $\sigma > 0$ such that, for any $u_c \in \ldiv(\Omega)$ which satisfies
 \begin{equation} \label{hyp.uc}
  \| {u_c}_{\rvert \Omega} \|_{L^2(\band)} \leq 3 \sigma,
 \end{equation}
 there exists a weak solution $u \in C^0([0,T];\ldiv(\Omega)) \cap L^2((0,T); H^1(\Omega))$ to \eqref{eq.ns} with the initial data $u_c$ and $f_g = 0$, which satisfies $u(T) = 0$.
\end{proposition}

Let us prove that the combination of these three propositions implies \cref{thm.main1}. First, thanks to standard arguments, it is sufficient to prove \cref{thm.main1} for initial data satisfying \rev{$\uinit\cdot\ex = 0$ on $\Gamma_0$ and $\Gamma_L$}. Indeed, applying vanishing boundary controls on the original system for any positive time guarantees that the state gains this property. Therefore, we assume that the initial data already satisfies this property. We fix the quantities step by step in the following manner.
\begin{itemize}
 \item Let $T > 0$ and $\uinit \in \ldiv(\Omega)$ satisfying \rev{$\uinit\cdot\ex = 0$ on $\Gamma_0$ and $\Gamma_L$}.
 \item Let $\rho_b > 0$ be given by \cref{prop_B} for a time interval of length $T/3$.
 \item Let $\sigma > 0$ be given by \cref{prop_C} for a time interval of length $T/3$.
 
 \item Let $\delta_1 \in (0,\frac{1}{2})$ small enough such that $\| {({\uinit})}_{\rvert \{ |y| \geq 1-2\delta_1\}} \|_{L^2(\Omega)} \leq \sigma$.
 
 
 
 \item Let $k \in \N$ and $\eta > 0$.
 \item Let $\eta_b := \min \{ \eta/2, \sigma, \eta / (2C_{k,\delta_1}) \}$.
 
 \item  We apply \cref{prop_A} with a time interval of length $T/3$, $\rho_b$, 
 $k$ and $\eta_b$. Hence, there exists $T_b \leq T/3$ and a solution $u$ defined on $[0,T_b]$ with $u(0)_{\rvert \Omega} = \uinit$ and such that $u_b := u(T_b)$ satisfies \eqref{eq:ub-estimate}, \eqref{ub.anal} and \eqref{ub.int}.
 
 \item We apply \cref{prop_B} with a time interval of length $T/3$, $\rho_b$, $\delta_1$, $k$ and $\sigma$. This yields a solution $u$ defined on $[T_b,T_c]$ with $T_c \leq T_b+T/3 \leq 2T/3$ such that $u_c := u(T_c)$ satisfies \eqref{eq:uc-estimate}.
 
 \item By triangular inequality, $\| {u_b}_{\rvert \{ |y| \geq 1-2\delta_1\}} \|_{L^2(\band)} \leq \| {u_b}_{\rvert \band\setminus\Omega} \|_{L^2(\band)} 
 + \| {u_b}_{\rvert\Omega} - {\uinit} \|_{L^2(\Omega)} + \| ({\uinit})_{\rvert \{ |y| \geq 1-2\delta_1\}} \|_{L^2(\Omega)}$. Since $\eta_b \leq \sigma$, \eqref{eq:ub-estimate} and \eqref{eq:uc-estimate} imply that \eqref{hyp.uc} holds.
 
 \item Finally, we apply \cref{prop_C}. This yields a solution $u$ defined on $[T_c,T_3]$ with $T_3 := T_c + T/3 \leq T$ such that $u(T_3) = 0$.
 
 \item The concatenated forces $f_c$ and $f_g$ are $C^\infty$ by parts in time with $C^\infty$ regularity in space.
 
 \item This concludes the proof of \cref{thm.main1} up to extending the solution and the forces by $0$ on $[T_3,T]$.
 
\end{itemize}

\begin{figure}[ht!]
 \centering
 \includegraphics{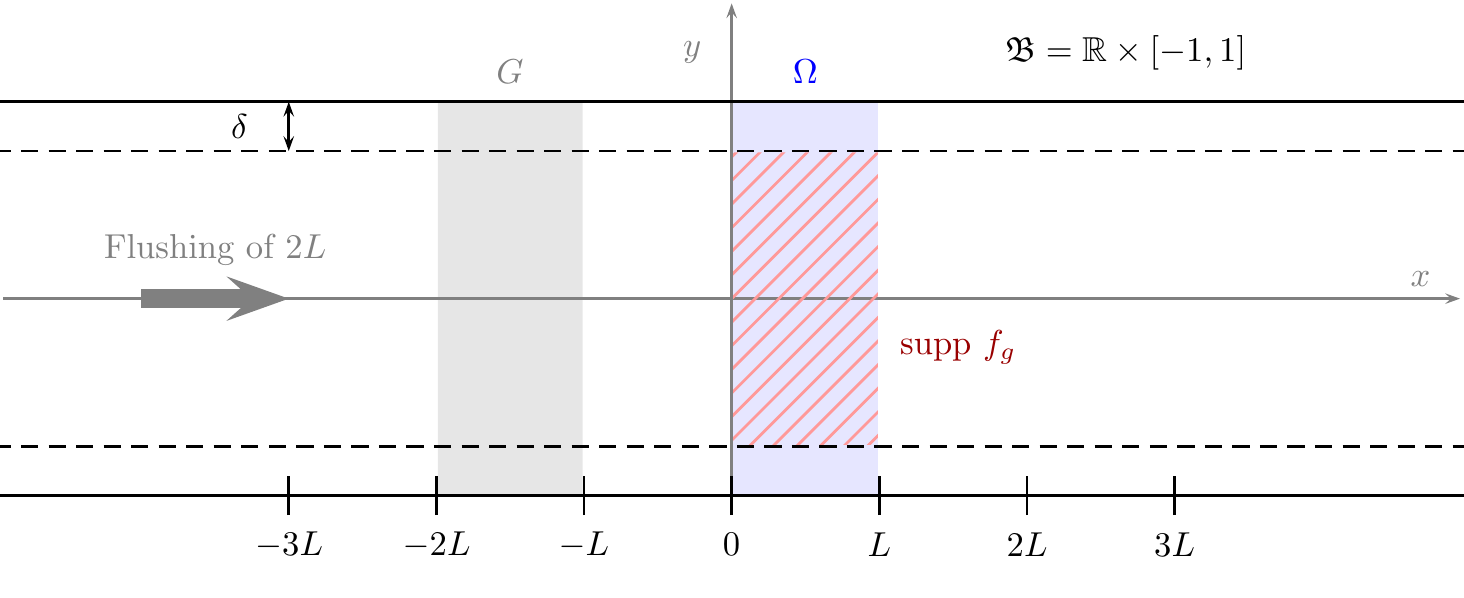}
 \caption{The horizontal band $\band = \mathbb{R} \times [-1,1]$, the physical domain $\Omega = [0,L]\times[-1,1]$ and the domain $G = [-2L,-L]\times[-1,1]$.}
 \label{fig:ghost}
\end{figure}

\begin{remark}
 The fact that, starting from a finite energy initial data, the solution to the Navier-Stokes equation instantly becomes analytic is well-known. However, in the uncontrolled setting, the analytic radius only grows like $\sqrt{t}$. In \cref{prop_A}, {we use the phantom force to enhance the regularization in short time.} 

 \rev{It would be desirable to know if \cref{prop_A} holds with a phantom force $f_g$ satisfying a support condition such as \eqref{f.support-rest} of \cref{prop_B}, since it would imply that \cref{thm.main1} holds with a phantom force whose support never touches the boundary $\Gamma_\pm$.}
\end{remark}

\begin{remark}
 The small-time global approximate null controllability result of \cref{prop_B} will be proved thanks to a return-method argument (see \cref{sec:strategy}). A base flow will shift the whole band~$\band$ of a distance $2L$ towards the right. Roughly speaking, the main part of the initial data ${u_b}_{\rvert\Omega}$ will then be outside of the physical domain and killed by a control force. However, since we need to work in an analytic setting (to establish estimates for a PDE with a derivative loss, see \cref{sec:r}), this action cannot be exactly localized outside of the physical domain. Its leakage inside the physical domain will be related to the values of $u_b$ in $G$ and lead to the unwanted phantom force. This explains estimate \eqref{f.small-rest}. Of course, since $u_b$ is analytic in the tangential direction $x$, it cannot satisfy ${u_b}_{\rvert G} = 0$.
\end{remark}

\rev{
\begin{remark} \label{rk:no_carleman}
 \cref{prop_C} is a direct consequence of known results concerning the small-time local null controllability of the Navier-Stokes equation (see \cref{rk:local} references). In our context involving phantom forces, one can avoid the use of these technical local results relying on Carleman estimates thanks to the following alternative statement, leveraging the phantom force to drive a small data $u_c$ exactly to zero: ``Let $T, \eta > 0$ and $k \in \N$. There exists $\sigma > 0$ such that, for any $u_c \in \ldiv(\Omega)$ satisfying \eqref{hyp.uc}, there exists a phantom force $f_g \in C^\infty([0,T]\times\bar\Omega)$ satisfying \eqref{f.small} and a weak solution $u \in C^0([0,T];\ldiv(\Omega)) \cap L^2((0,T);H^1(\Omega))$ to \eqref{eq.ns} with the initial data $u_c$, which satisfies $u(T) = 0$.'' The sketch of proof of this result would be as follows. Let $\bar u$ be the solution to the free Navier-Stokes equation on $[0,T]$ starting from $\bar{u}(0) = u_c$. One constructs explicitly $u$ as $u(t) := \bar{u}(t)$ for $t \leq \frac T 2$ and $u(t) := \beta(t) \bar{u}(\frac T 2)$, where $\beta \in C^\infty([\frac T 2,T];[0,1])$ with $\beta(\frac T 2) = 1$ and $\beta(T) = 0$. Plugging this explicit formula in \eqref{eq.ns} gives a definition of the phantom force $f_g$ that one needs to use. Hence, there exists $C_{k,\beta} > 0$ such that, if $u_c$ is small enough, then $\|f_g\|_{L^1((0,T);H^k(\Omega))} \leq C_{k,\beta} \| \bar{u}(\frac T 2) \|_{H^{k+2}(\Omega)}$. Moreover, thanks to classical regularization result for the 2D Navier-Stokes equation, there exists $C_k(T) > 0$ such that $\| \bar{u}(\frac T 2) \|_{H^{k+2}(\Omega)} \leq C_k(T) \| u_c \|_{L^2(\Omega)}$. Hence, choosing $3 \sigma C_k(T) C_{k,\beta} \leq \eta$ concludes the proof. 
\end{remark}
}

\cref{prop_A} is proved in \cref{sec:local}. \cref{prop_B} is the main contribution of this paper. We explain our strategy to prove this small-time global approximate null controllability result in \cref{sec:strategy}. It is based on the construction of approximate trajectories and the \emph{well-prepared dissipation method}. We give estimates concerning the boundary layer in \cref{sec:dissipation}, estimates concerning the remainder in \cref{sec:r} and finally analytic estimates on the approximate trajectories themselves as an appendix in \cref{sec:sources}.

\section{Regularization enhancement}
\label{sec:local}

This section is devoted to the proof of \cref{prop_A}.

\subsection{Fourier analysis in the tangential direction}

We introduce a few notations that will be used throughout this paper.

\paragraph{Tangential Fourier transform.} To perform analytic estimates in the tangential direction, we use Fourier analysis in the tangential direction. For $a \in L^2(\band)$, we will denote its Fourier transform with respect to the tangential variable by $\mathcal{F}a$ and define it as
\begin{equation} \label{eq:FOURIER}
 \mathcal{F}a (\xi, y) := \int_{\R_x} a(x,y) e^{-i x \xi} \dd x.
\end{equation}
We define similarly the reciprocal Fourier transform $\mathcal{F}^{-1}$, which obviously also acts only on the tangential variable.

\newcommand{\pn}{\mathbf{P}_N}

\paragraph{Band-limited functions.} Let $N > 0$. We will sometimes need to consider functions in $L^2(\band)$ whose Fourier transform is supported within the set of tangential frequencies~$\xi$ satisfying $|\xi| \leq N$. Therefore, we introduce the Fourier multiplier
\begin{equation} \label{PN}
 \pn(\xi) := \mathbf{1}_{[-N,N]}(\xi)
\end{equation}
and the associated functional space
\begin{equation}
 L^2_N(\band) := \left\{ a \in L^2(\band); \enskip a = \pn a \right\}.
\end{equation}
For any $k\in\N$ and $a\in H^k(\band)$, it is clear that
\begin{equation} \label{pnhk}
 \| \pn a -  a \|_{H^k(\band)} \to 0
 \quad \text{as} \quad 
 N \to +\infty.
\end{equation}

\subsection{Regularization enhancement using the phantom and the control}
 
We start with the following lemma concerning the possibility to remove high tangential frequencies from a smooth initial data. We denote by $\mathbb{P}$ the usual Leray projector on divergence free vector fields, tangent to $\partial\band$.

\begin{lemma} \label{thm:secours}
 There exists a geometric constant $C_\theta > 0$ such that the following result holds. Let $u_a \in \ldiv(\band)$ and $T_b > 0$. We denote by $\widetilde{u}_a$ the solution of the free Navier-Stokes equation at time $T_b$ starting from $u_a$.
 There exists a family indexed by $N > 0$ of vector fields $u_N \in C^0([0,T_b];\ldiv(\band))\cap L^2((0,T_b);H^1(\band))$ associated with forces $f_N \in C^0([0,T_b];H^k(\band))$, which are weak Leray solutions to
 \begin{equation}
  \label{good1}
   \partial_t u_N - \Delta u_N + \mathbb{P} [ (u_N\cdot\nabla) u_N ] = \mathbb{P} f_N,
   \quad
   u_N(0) = u_a
 \end{equation}
 and satisfy, for any $k\in\N$ and $\rho_b > 0$,
 \begin{gather}
  \label{good2}
  \| (f_N)_{\rvert \Omega} \|_{L^1((0,T_b);H^k(\Omega))} \underset{N\to+\infty}{\longrightarrow} 0,
  \\
  \label{good3}
  f_N \in C^\infty([0,T_b]\times\band),
  \\
  \label{good4_omega}
  \| (u_N(T_b) - \widetilde{u}_a)_{\rvert \Omega} \|_{L^2(\Omega)} \underset{N\to+\infty}{\longrightarrow} 0,
  \\
  \label{good4_G}
  \| (u_N(T_b))_{\rvert G} \|_{H^k(G)} \underset{N\to+\infty}{\longrightarrow} 0,
  \\
  \label{good4bis}
  \| (u_N(T_b))_{\rvert \band\setminus\bar{\Omega}} \|_{L^2(\band\setminus\bar\Omega)}
  \leq 
  C_\theta \| (\widetilde{u}_a)_{\rvert \band\setminus\bar{\Omega}} \|_{L^2(\band\setminus\bar\Omega)}
  + \underset{N\to+\infty}{o}(1),
  \\
  \label{good5}
  \exists C_N > 0, \quad 
  \sup_{m\geq 0} \| \partial_x^m u_N(T_b) \|_{H^3(\band)}
  \leq \frac{m!}{\rho_b^m} C_N,
  \\
  \label{good6}
  \exists C_N > 0, \quad 
  \sum_{0\leq\alpha+\beta\leq 3} \|\p_x^\alpha \p_y^\beta u_N(T_b)\|_{L^1_x(L^2_y)} \leq C_N.
 \end{gather}
\end{lemma}

\begin{proof}
 Let $u_a \in \ldiv(\band)$. Let $T_b > 0$. Let $v$ be the weak Leray solution to
 \begin{equation}
  \partial_t v - \Delta v + \mathbb{P}[(v\cdot\nabla)v] = 0,
  \quad 
  v(0) = u_a.
 \end{equation}
 Hence, by definition, $\widetilde{u}_a = v(T_b)$. It is classical to prove that $v \in C^\infty((0,T_b]\times\band)$ (see e.g.~\cite{MR645638}).
 Let $\beta \in C^\infty([0,T_b];[0,1])$ with $\beta = 1$ on $[0,T_b/3]$ and $\beta = 0$ on $[2T_b/3,T_b]$.
 Let $\theta \in C^\infty(\band ;[0,1])$ with $\theta = 1$ for $x \in [0,L]$ and $\theta = 0$ for $x < -L$ or $x > 2L$. Let
 \begin{equation} \label{def:psi.v}
   \psi^v(t,x,y) := - \int_{-1}^y v(t,x,y) \cdot e_x \dd y.
 \end{equation}
 We introduce
 \begin{equation}
   u := \nabla^\perp \left( \beta \psi^v + (1-\beta) \theta \psi^v \right).  
 \end{equation}
 Hence $\diverg u = 0$. Moreover, since $\beta$ does not depend on the space variables, one has
 \begin{equation}
   u = \beta v + (1-\beta) \theta v + (1-\beta) \psi^v \nabla^\perp \theta.
 \end{equation}
 Then, $u$ is the weak Leray solution to
 \begin{equation}
  \partial_t u - \Delta u + \mathbb{P}[(u\cdot\nabla)u] = \mathbb{P} g,
  \quad 
  u(0) = u_a,
 \end{equation}
 where we set
 \begin{equation}
  \begin{split}
   g  := \dot{\beta} (1-\theta) v
   & - \dot \beta \psi^v \nabla^\perp \theta + (1-\beta) \partial_t \psi^v \nabla^\perp \theta
   \\ & - 2 (1-\beta) (\nabla \theta \cdot \nabla) v - (1-\beta)\Delta \theta v 
   - (1-\beta) \Delta \left( \psi^v \nabla^\perp \theta \right)
   \\
   &  - \beta (1-\beta) (v\cdot\nabla)((1-\theta)v)
   - (1-\beta)^2(\theta v \cdot\nabla)((1-\theta)v)
   \\
   & 
   + (1-\beta) (\beta + (1-\beta)\theta) (v\cdot\nabla) (\psi^v \nabla^\perp \theta)
   + (1-\beta) \psi^v (\nabla^\perp \theta \cdot \nabla) u.
  \end{split}
 \end{equation}
 Since each term involves, on the one hand $\dot \beta$ or $1-\beta$ and, on the other hand, $1-\theta$ or a derivative of $\theta$, $\supp g \subset [T_b/3,T_b]\times(\band\setminus{\bar{\Omega}})$ and $g \in C^\infty([0,T_b]\times\band)$. We define
 \begin{equation}
  \label{uM.new}
  u_N(t) := \beta(t) u(t) + (1-\beta(t)) \mathbf{P}_N u(t),
 \end{equation}
 where $\PN$ is defined in \eqref{PN}. In particular, \eqref{uM.new} implies \eqref{good1}, provided that one sets  
 \begin{equation} \label{FM}
  \begin{split}
   f_N := \dot{\beta} (u-\PN u) 
   & + (1-\beta) \Big( (\PN u \cdot \nabla) \PN u - \PN ((u\cdot\nabla)u) \Big)
   \\
   & + \beta (1-\beta) \Big(
    (u \cdot\nabla) (\PN u-u) 
    - (\PN u  \cdot\nabla) (\PN u-u) 
   \Big)
   \\ 
   & + \beta g + (1-\beta) \PN g.
  \end{split}
 \end{equation}
 Let $k\in\N$. Thanks to definition \eqref{FM}, there holds \eqref{good2} and \eqref{good3}. Indeed, $u$ belongs to $C^0([T_b/3,T_b];H^{k+1}(\Omega))$ and the family $\PN u$ converges towards $u$ in this space. From \eqref{uM.new}, at the final time, one has 
 \begin{equation} \label{umb}
  u_N(T_b) = \PN u(T_b) = 
  \PN \left(\theta v(T_b) + \psi^v(T_b) \nabla^\perp \theta\right) 
  = 
  \PN \left(\theta \widetilde{u}_a + \psi^{\widetilde{u}_a} \nabla^\perp \theta\right),
 \end{equation}
 where $\psi^{\widetilde{u}_a}$ is defined from $\widetilde{u}_a$ similarly as $\psi^v$ is from $v$ in \eqref{def:psi.v}.
 In particular, we deduce from \eqref{umb} that $u_N(T_b)$ is entire in $x$ so that, for any $\rho_b > 0$, there exists $C_b >0$ such that \eqref{good5} holds. We also deduce from \eqref{umb} that $u_N(T_b) \to \theta\widetilde{u}_a {+ \psi^{\widetilde{u}_a} \nabla^\perp \theta}$ in $H^k(\band)$, which implies \eqref{good4_omega}, \eqref{good4_G} and \eqref{good4bis} with $C_\theta = 2 \| \theta \|_{W^{1,\infty}(\band)}$.

 \bigskip
 
 To obtain \eqref{good6}, we change slightly the definition \eqref{PN} of $\PN$. Instead of a rectangular window filter, we define $\PN$ as the Fourier multiplier $W_N$, where $W_N \in C^\infty(\R;[0,1])$ is such that $W_N(\xi) = 1$ for $\xi \in [-N+1,N-1]$ and $W_N(\xi) = 0$ when $|\xi|\geq N$. This preserves the property \eqref{pnhk} but has a better behavior with respect to $L^1$ norms in space. Indeed, if $\phi \in \mathcal{S}(\R,\R)$, one checks that $W_N \phi \in L^1(\R)$. This property implies \eqref{good6} because $u(T_b)$ has a compact support.
\end{proof}

\subsection{Proof of the regularization proposition}

We turn to the proof of \cref{prop_A}. 

Let $T,\rho_b,\eta_b > 0$ and $k\in\N$. Let $\uinit \in \ldiv(\Omega)$ satisfying \rev{$\uinit\cdot\ex = 0$ on $\Gamma_0$ and $\Gamma_L$. 
Let $u_a$ be the extension by $0$ of $\uinit$ to $\band$. Since $\uinit \cdot \ex = 0$ on $\Gamma_0$ and $\Gamma_L$, $u_a \in \ldiv(\band)$.} For $T_b \in( 0,T)$ small enough, the free solution starting from $u_a$ at time $T_b$, say $\widetilde{u}_a$ satisfies $\|(\widetilde{u}_a)_{\rvert \band\setminus\bar{\Omega}}\|_{L^2(\band\setminus\bar{\Omega})} \leq \eta_b / 2 C_\theta$ and
$\|(\widetilde{u}_a)_{\rvert\Omega} - \uinit\|_{L^2(\Omega)} \leq \eta_b / 5$.

We choose  $N$ large enough such that \eqref{good4_omega}, \eqref{good4_G} and \eqref{good4bis} imply that  $u_b := u_N(T_b)$, where the family $(u_N,f_N)$ is  given by \cref{thm:secours},  satisfies \eqref{eq:ub-estimate} and such that  \eqref{good2} ensures \eqref{fa.small-rest}. Estimates \eqref{good5} and \eqref{good6} prove \eqref{ub.anal} and \eqref{ub.int}. This concludes the proof of \cref{prop_A}, provided that we define $f_g := (f_N)_{\rvert\Omega}$ and $f_c := (f_N)_{\rvert \band\setminus \Omega}$, each being smooth within its support.

\section{Strategy for global approximate controllability}
\label{sec:strategy}

We explain our strategy to prove \cref{prop_B}. Let $T > 0$, $u_b \in \ldiv(\band)$, $\delta \in (0,\f12)$ and $k\in \N$. We intend to construct a family of approximate trajectories depending on a small parameter $0 < \varepsilon \ll 1$ and driving $u_b$ approximately to zero. We detail the construction of this family in the following paragraphs. Then, we prove estimates on boundary layer terms for these approximate trajectories in \cref{sec:dissipation}. We prove estimates on the remainder in \cref{sec:r} and postpone analytic-type estimates for these approximate trajectories to \cref{sec:sources}.

\subsection{Small-time to small-viscosity scaling}

Let $\kappa \in (0,1)$. Although it might seem like a further complication, our strategy is based on trying to control the system \eqref{eq:ns-fc-fg} at an even shorter time scale, $\varepsilon^{1-\kappa} T$, passing through intermediate states (velocities) of order $1/\varepsilon$. For $\varepsilon \in (0,1)$, we introduce the trajectories
\begin{gather}
 \label{eq:UP}
 U^\varepsilon(t,x,y) := \ue(t/\varepsilon,x,y)/\varepsilon, \quad P^\varepsilon(t,x,y) := p^\varepsilon(t/\varepsilon,x,y)/\varepsilon^2, \\ 
 \label{eq:FC-FG}
 F_g^\varepsilon(t,x,y):=f_g^\varepsilon(t/\varepsilon,x,y)/\varepsilon^2 \quad \text{and} \quad F_c^\varepsilon(t,x,y):=f_c^\varepsilon(t/\varepsilon,x,y)/\varepsilon^2. 
\end{gather}
The tuples $(U^\varepsilon,P^\varepsilon,F_c^\varepsilon,F_g^\varepsilon)$ define solutions to \eqref{eq:ns-fc-fg} with initial data $u_b$ if and only if the new unknowns $(\ue,p^\varepsilon,f_c^\varepsilon,f_g^\varepsilon)$ are solutions to the rescaled system
\begin{equation} \label{eq:nse}
 \left\{
 \begin{aligned}
   \partial_t \ue + \left( \ue \cdot \nabla \right) \ue + \nabla p^\varepsilon 
   - \varepsilon \Delta \ue & = f_g^\varepsilon +  f_c^\varepsilon 
   &&\quad \textrm{in } (0, T/\varepsilon^\kappa) \times \band, \\
   \diverg \ue & = 0
   &&\quad \textrm{in } (0, T/\varepsilon^\kappa) \times \band, \\
  \ue & = 0   &&\quad \textrm{on } (0, T/\varepsilon^\kappa) \times \partial\band, \\
   \ue \rvert_{t = 0} & = \varepsilon u_b
   &&\quad \textrm{in } \band.
 \end{aligned}
 \right.
\end{equation}
Observe the three differences between \eqref{eq:nse} and the original system \eqref{eq:ns-fc-fg}: 
\begin{itemize}
\item the Laplace term has a small factor $\varepsilon$ in front of it rather than $1$, 
\item the system is set on the long time interval $(0, T/\varepsilon^\kappa)$ rather than $(0, T)$, 
\item the initial data is $ \varepsilon u_b$ rather than $u_b$. 
\end{itemize}
We construct approximate solutions to \eqref{eq:nse} in the following paragraph.

\subsection{Return method ansatz}

We introduce the following explicit approximate solution to \eqref{eq:nse}:
\begin{align}
 \label{eq:uapp}
 \uapp(t,x,y) & := u^0(t) + \chi(y) v^0\left(t,\varphi(y)/\sqrt{\varepsilon}\right) + \varepsilon u^1(t,x,y) {+ \e^2 w^\e(t,y)}, \\
 p^\varepsilon_{\mathrm{app}}(t,x,y) & := p^0(t,x), \\
 \label{eq:fc}
 f^\varepsilon_c(t,x,y) & := \varepsilon f^1_{\rvert \band \setminus \Omega}(t,x,y), \\
 \label{eq:fg}
 f^\varepsilon_g(t,x,y) & := \varepsilon f^1_{\rvert \Omega}(t,x,y).
\end{align}
In the following lines, we define each of the terms involved in this approximate solution. We refer to \cref{sec:why} for comments on the choice of these profiles.

\paragraph{Base Euler flow profile.}

Let $n \in \N$ satisfying $n \geq 3$ and
\begin{equation} \label{n.kappa}
 n > \frac{3}{4} \left( \frac{1}{\kappa} - 1 \right).
\end{equation}
Let $h$ in $C^\infty(\R_+, \R)$ be such that 
\begin{gather}
 \label{hyp.h.support}
 \supp~h \subset (0,T/3] \cup [2T/3,T), \\
 \label{hyp.h.flow}
 \int_0^{T/3} h(t) \dt = 2L, \\
  \label{hyp.h.moments}
 \int_0^T t^k h(t) \dt = 0
 \quad \text{for} \quad 
 0 \leq k < n.
\end{gather}
We define
\begin{equation} \label{eq:u0p0}
 u^0(t) := h(t) \ex 
 \quad \text{and} \quad
 p^0(t,x) := - \dot{h}(t) x.
\end{equation}
For a function $a \in L^2(\band)$, we will denote its translation along the base flow $h$ by
\begin{equation} \label{trans}
 (\tau_h a) (t,x,y) := a\left(x - \int_0^t h(s) \dd s, y\right). 
\end{equation}

\paragraph{Boundary layer profile.}

Let $\varphi \in C^\infty([-1,1],[0,1])$ such that $\varphi(\pm 1) = 0$ and $|\varphi'(y)| = 1$ for $|y| \geq 1/4$. Let $\chi \in C^\infty([-1,1],[0,1])$ such that $\chi(y) = 1$ for $|y| \geq 2/3$, 
and $\chi(y) = 0$ for $|y| \leq 1/3$. Let $V(t,z)$ be the solution to
\begin{equation} \label{eq.V}
 \left\{
 \begin{aligned}
   \partial_t V - \partial_{zz} V & = 0 && \quad \text{in } \R_+ \times \R_+, \\
   V(t,0) & = h(t) &&  \quad \text{on } \R_+, \\
   V(0,z) & = 0 &&  \quad \text{in } \R_+.
 \end{aligned}
 \right.
\end{equation}
We define
\begin{equation}
 v^0(t,z) := - V(t,z) \ex.
\end{equation}
In the sequel, for any function $\mathcal{V}(t,z)$, depending on the fast variable, we will denote its evaluation at $z = \varphi(y) /\sqrt{\e}$ by
\begin{equation} \label{eval}
 \eval{\mathcal{V}}(t,y) := \mathcal{V}\left(t,\frac{\varphi(y)}{\sqrt{\e}}\right).
\end{equation}

\paragraph{Linearized Euler flow profile.}

Let $\beta \in C^\infty(\R_+,[0,1])$ non-increasing such that $\beta(t) = 1$ for $t \leq T/3$ and $\beta(t) = 0$ for $t \geq 2T/3$. Let $\chi_\delta \in C^\infty([-1,1],[0,1])$ such that $\chi_\delta(y)=1$ for $|y|\leq 1-2\delta$ and $\chi_\delta(y)=0$ for $|y|\geq 1-\delta$. We define the stream function associated with $u_b$, then $u^1$ and eventually the force $f^1$:
\begin{align}
 \label{psi_b}
 \psi_b(x,y) & := - \int_{-1}^y u_b(x,y') \cdot \ex \dd y', \\
 \label{eq:magic-u1}
 u^1(t,x,y) & := \beta(t) \tau_h \nabla^\perp [\chi_\delta \psi_b] + \tau_h \nabla^\perp[(1-\chi_\delta)\psi_b], \\
 \label{eq:magic-f1}
 f^1(t,x,y) & := \dot{\beta}(t) \left(\chi_\delta(y) u_b(x-2L,y) - \chi_\delta'(y) \psi_b(x-2L,y) \ex \right).
\end{align}

\paragraph{Technical profile.}
For $t \in \R_+$ and $y \in [-1,1]$, we define the source
\begin{equation} \label{eq.FW}
 f_W^\e := - \frac{\chi''}{\varphi^2} \eval{z^2 V} 
 - 2 \frac{\chi'}{\varphi^3} \varphi' \eval{z^3 \partial_z V}.
\end{equation}
Let $W^\e(t,y) : \R_+ \times [-1,1] \to \R$ be the solution to
\begin{equation} \label{eq.W}
  \left\{
 \begin{aligned}
   \partial_t W^\e - \e \partial_{yy} W^\e & = 
   f_W^\e
   && \quad \text{in } \R_+ \times [-1,1], \\
   W^\e(t,\pm 1) & = 0 &&  \quad \text{on } \R_+, \\
   W^\e(0,y) & = 0 &&  \quad \text{in } [-1,1].
 \end{aligned}
 \right.
\end{equation}
Finally, we let 
\begin{equation} \label{eq.wW}
 w^\e(t,y) := W^\e(t,y) \ex.
\end{equation}

\paragraph{Equation satisfied by the approximate trajectories.}

Then $(\uapp,p^\varepsilon_\mathrm{app})$ are solutions to
\begin{equation} \label{eq:nse-app}
 \left\{
 \begin{aligned}
   \partial_t \uapp + \left(\uapp \cdot \nabla \right) \uapp
   - \varepsilon \Delta \uapp + \nabla p^\varepsilon_\mathrm{app} & =   f_c^\varepsilon + f_g^\varepsilon
   + \varepsilon f^\varepsilon_\mathrm{app}
   &&\quad \textrm{in } (0, T/\varepsilon^\kappa) \times \band, \\
   \diverg \uapp & = 0
   &&\quad \textrm{in } (0, T/\varepsilon^\kappa) \times \band, \\
\uapp & = 0   &&\quad \textrm{on } (0, T/\varepsilon^\kappa) \times \partial\band, \\
   \uapp \rvert_{t = 0} & = \varepsilon u_b
   &&\quad \textrm{in } \band,
 \end{aligned}
 \right.
\end{equation}
where we define
\begin{equation} \label{eq:fapp}
 \begin{split}
  f^\varepsilon_\mathrm{app}
 := &
 - \varepsilon \Delta u^1 + \varepsilon (u^1 \cdot \nabla) u^1
 {+\e^2 W^\e \partial_x u^1 + \e^2 (u^1 \cdot \ey) \partial_y w^\e}
 \\ & - \chi \eval{V} \partial_x u^1
 - \frac{u^1 \cdot \ey}{\varphi} \left(\sqrt{\varepsilon} \chi' \eval{z V} + \chi \varphi' \eval{z \partial_z V}\right) \ex.
 \end{split}
\end{equation}

\subsection{Estimates and proof of approximate controllability}

By construction, the approximate trajectory will be small at the final time.
\begin{proposition} \label{thm:uapp-estimate}
 There exists a constant $C_{\mathrm{app}} > 0$ such that, for $\e>0$ small enough,
 \begin{equation} \label{eq:uapp-small}
  \begin{split}
   \frac{1}{\e} \| \uapp(T/\varepsilon^\kappa&)_{\rvert\Omega} \|_{L^2(\band)} 
   \\ 
   & \leq C_{\mathrm{app}} \left( \e^{\frac14} + \e^{\kappa(n-\frac{3}{4}(\frac{1}{\kappa}-1))} |\ln\e|^{n+\frac{3}{4}}\right) + \| {u_b}_{\rvert \{ |y| \geq 1-2\delta \} } \|_{L^2(\band)}.
  \end{split}
 \end{equation}
\end{proposition}

Moreover, the approximate trajectory can be arbitrarily close to a true trajectory. Indeed, we can construct a remainder which is small, provided that the initial data $u_b$ is sufficiently regular (its tangential analytic radius is large enough).

\begin{proposition} \label{thm:r-estimate}
 There exists $\rho_b > 0$, depending only on $T$ such that, if $u_b$ satisfies \eqref{ub.anal} and \eqref{ub.int} for some $C_b > 0$, there exists $C_r > 0$ such that, for $\e > 0$ small enough, there exists a weak Leray solution $r^\e \in C^0([0,T/\e^\kappa],\ldiv(\band))\cap L^2((0,T/\e^\kappa),H^1(\band))$ to
 \begin{equation} \label{eq.nse.r}
 \left\{
 \begin{aligned}
   \partial_t r^\varepsilon + \left( \uapp \cdot \nabla \right) r^\varepsilon 
   +  \e \left( r^\e \cdot \nabla \right) r^\varepsilon
   +   \left( r^\varepsilon \cdot \nabla \right)  \uapp &
   \\
   - \varepsilon \Delta r^\varepsilon  + \nabla \pi^\varepsilon  & = - f^\varepsilon_{\mathrm{app}}
   &&\quad \textrm{in } (0, T/ \varepsilon^\kappa) \times \band, \\
   \diverg r^\varepsilon & = 0
   &&\quad \textrm{in } (0, T/\varepsilon^\kappa) \times \band, \\
   r^\varepsilon & = 0 
   &&\quad \textrm{on } (0, T/ \varepsilon^\kappa) \times \partial\band, \\
   r^\varepsilon \rvert_{t = 0} & = 0
   &&\quad \textrm{in } \band,
 \end{aligned}
 \right.
 \end{equation}
 which moreover satisfies
 \begin{equation} \label{eq:r-small}
  \| r^\varepsilon \|_{L^\infty((0,T/\e^\kappa);L^2(\band))} \leq C_r (\e^{\frac14} + \e^{1-\kappa}).
 \end{equation}
\end{proposition}

It is straightforward to check that \cref{thm:uapp-estimate} and \cref{thm:r-estimate} imply \cref{prop_B}. Indeed, let $\rho_b$ be given by \cref{thm:r-estimate} and assume that $u_b$ satisfies \eqref{ub.anal} and \eqref{ub.int} for some $C_b > 0$. We choose $\varepsilon > 0$ small enough such that the conclusions of both propositions hold. We construct an exact trajectory by setting
\begin{equation} \label{app.exact}
 \ue := \uapp + \e r^\e
 \quad \text{and} \quad
 p^\e := p^\e_{\mathrm{app}} + \e \pi^\e.
\end{equation}
Combining \eqref{app.exact} with the equation~\eqref{eq:nse-app} satisfied by $\uapp$ and the equation \eqref{eq.nse.r} satisfied by $r^\e$ proves that $\ue$ is a weak solution to \eqref{eq:nse}.

Let $\sigma > 0$. Choosing $\e > 0$ small enough, summing estimates \eqref{eq:uapp-small} and \eqref{eq:r-small} and recalling the definition \eqref{app.exact} of $u^\e$, the assumption \eqref{n.kappa} on $n$ and the scaling \eqref{eq:UP} proves that \eqref{eq:uc-estimate} holds at the time $T_c := \varepsilon^{1-\kappa}T < T$.

Since $u_b$ satisfies \eqref{ub.anal}, $u_b \in C^\infty(\band)$. Thus, thanks to \eqref{eq:FC-FG}, \eqref{eq:fc}, \eqref{eq:fg} and \eqref{eq:magic-f1}, $F_g \in C^\infty([0,T_c]\times\bar{\Omega})$, $F_c \in C^\infty([0,T_c]\times\band\setminus\bar{\Omega})$ and moreover
\begin{gather}
 \supp F_g \subset (0,T_c) \times [0,L]\times[-1+\delta,1-\delta],
 \\
 \supp F_c \subset (0,T_c) \times \band\setminus\bar{\Omega}, 
\end{gather}
Moreover, using \eqref{eq:FC-FG}, \eqref{eq:fg} and \eqref{eq:magic-f1}, one has
\begin{equation}
 \begin{split}
  \| F^\varepsilon_g(t) \|_{L^1([0,T_c];H^k(\Omega))} 
  & = \frac{1}{\varepsilon^2} \| f^\varepsilon_g(t/\varepsilon) \|_{L^1([0,\e^{1-\kappa}T];H^k(\Omega))} \\
  & = \frac{1}{\varepsilon} \| f^\varepsilon_g(t) \|_{L^1([0,T/\varepsilon^\kappa];H^k(\Omega))}
  \\
  & = \| f^1_{\rvert\Omega} \|_{L^1([0,T/\e^\kappa];H^k(\Omega))}
  \\
  & \rev{\leq} \| \chi_\delta u_b + \chi_\delta' \psi_b \|_{H^k(G)},
 \end{split}
\end{equation}
where we recall that the set $G$ is defined in \eqref{def-G}. This proves the estimate \eqref{f.small-rest} concerning the size of the phantom force, for a constant $C_{k,\delta}$ which only depends on the norm of $\chi_\delta$ in $H^{k+1}(-1,1)$, and thus concludes the proof of the approximate controllability result \cref{prop_B}.

\bigskip

We prove \cref{thm:uapp-estimate} in \cref{sec:dissipation} (thanks to the well-prepared dissipation method) and \cref{thm:r-estimate} in \cref{sec:r} (using a long-time nonlinear Cauchy-Kovalevskaya estimate).

\subsection{Comments and insights on the proposed expansion}
\label{sec:why}

\begin{remark}[Return method and base Euler flow]
 Since system \eqref{eq:nse} can be seen as a perturbation of the Euler equations a natural idea is to follow the return method introduced by Coron in~\cite{MR1164379} (see also~\cite[Chapter 6]{MR2302744}) to prove the controllability of the Euler equations in the 2D case (see also~\cite{MR1380673}, and \cite{MR1745685} for the 3D case). Loosely speaking the idea is to overcome that the linearized problem around zero is not controllable by introducing, thanks to the boundary control, a velocity $u^0$ of order $O(1)$  (whereas the initial velocity is only of order $O(\varepsilon)$) solution to the Euler equation satisfying  $u^0  \rvert_{t = 0} = u^0 \rvert_{t = T} = 0$ and such that the corresponding flow flushes all the domain out during the time interval $(0,T)$. In the present case of a rectangle this step is pretty easy and explicit: it corresponds to the introduction of a flow which flushes out the initial data. From \eqref{eq:u0p0}, we get that $(u^0,p^0)$ indeed solves the incompressible Euler equation:
\begin{equation} \label{eq:euler-u0}
 \left\{
 \begin{aligned}
   \partial_t u^0 + \left( u^0 \cdot \nabla \right) u^0 
   & = - \nabla p^0,
   &&\quad \textrm{in } \R_+ \times \band, \\
   \diverg u^0 & = 0
   &&\quad \textrm{in } \R_+ \times \band, \\
   u^0 \cdot \ey & = 0 
   &&\quad \textrm{on } \R_+ \times \partial\band,
 \end{aligned}
 \right.
\end{equation}
 with initial data $u^0(0) = 0$ and $u^0(t) = 0$ for $t \geq T$.
\end{remark}

\begin{remark}[Transport of the initial data]
 The term~$u^1$ takes into account the initial data~$u_b$, which is transported by the flow~$u^0$. Using \eqref{eq:u0p0}, \eqref{eq:magic-u1} and \eqref{eq:magic-f1}, we obtain that $u^1$ solves
 \begin{equation} \label{eq.u1}
 \left\{
 \begin{aligned}
   \partial_t u^1 + h(t) \partial_x u^1 & = f^1 
   && \quad \textrm{in } \R_+ \times \band , \\
   \diverg u^1 & = 0 
   && \quad \textrm{in } \R_+ \times \band , \\
   u^1  & = 0 
   && \quad \textrm{on } \R_+ \times \partial\band , \\
   u^1(0)  & = u_b 
   && \quad \textrm{in } \band. 
 \end{aligned}
 \right.
 \end{equation}
 Thanks to assumption~\eqref{hyp.h.flow}, it is clear that the initial data will be flushed outside of the domain at time $T/3$. During the time interval $[T/3,2T/3]$, the initial data $u_b$ has been shifted towards the right of a distance $2L$. This is the time interval during which the force $f^1$ kills most of the initial data (for $|y|\leq 1-\delta$). 
 
 The key point is that, outside of the physical domain, this force is merely a control. However, since we need this force to be analytic, it also acts a little bit within the physical domain. This gives rise to an unwanted phantom force.
\end{remark}

\begin{remark}[Boundary layer correction]
\label{remark-blc}
 A major difficulty is linked to the discrepancy between the Euler and the Navier-Stokes equations in the vanishing viscosity limit. Indeed, although inertial forces prevail inside the domain, viscous forces play a crucial role near the uncontrolled boundary, and give rise to a boundary layer of order $O(1)$ associated with the velocity $u^0$ which does not satisfy the tangential part of the Dirichlet condition on the top and bottom boundaries. 
 
 The purpose of the second term~$v^0$ is to recover the Dirichlet boundary condition by introducing the boundary layer generated by~$u^0$. Thanks to our previous choice of $u^0$ we will avoid the difficulty usually associated with the Prandtl equation. Indeed the boundary layer will also be fully horizontal (tangential) and will not depend on $x$ so that the equation for $v^0$ will deplete into a linear heat equation with non-homogeneous Dirichlet data depending on $u^0$. The quantity $\varphi(y) / \sqrt{\varepsilon}$ reflects quick variations within the boundary layer, where $\varphi(y)$ is the distance to the boundary.
\end{remark}


\section{Well-prepared dissipation method for the boundary layer}
\label{sec:dissipation}

The key argument of the well-prepared dissipation method is that the normal
dissipation involved in fluid mechanics boundary layer equations can dissipate
most of their energy, provided that the created boundary layers are 
``well-prepared'' in some sense. Roughly speaking, this preparation amounts
to ensure that they do not contain energy at low frequencies.

\subsection{Large time decay of the boundary layer profile}

In the work~\cite{2016arXiv161208087C} concerning the case of the Navier slip-with-friction
boundary condition, we used boundary controls to import enough vanishing moments
thanks to the transport by the Euler flow within the boundary layer.
In this work, we cannot use this strategy because we do not want the boundary
layer profile to depend on the slow tangential variable, see Remark \ref{remark-blc}.
Instead we rely on the assumptions  \eqref{hyp.h.moments} 
on the base Euler flow. 
We prove below that these conditions entail  a good decay for the boundary layer profile. 
This decay will be used both to prove
that the source terms generated by~$v$ in equation~\eqref{eq.nse.r} for the 
remainder are integrable with respect to time and that the boundary profile at
the final time is small enough to apply a local controllability result.
For $s, m \in \N$ and $I$ an interval of~$\R$, 
we introduce the following weighted Sobolev spaces:
\begin{equation} \label{def.Hsm}
 H^{s,m}(I) := \left\{ f \in H^{s}(I), \enskip
 \sum_{\alpha = 0}^{s} \int_I \left(1+z^2\right)^m \left|f^{(\alpha)}(z)\right|^2 \dz
 < + \infty \right\},
\end{equation}
which we endow with their natural norm. We will use this definition with $I = \R$
or $I = \R_+$.

\begin{lemma} \label{Lemma:V}
 Let $T > 0$, $s,n \in \N$ and $h \in C^\infty(\R,\R)$ 
 satisfying \eqref{hyp.h.support} and~\eqref{hyp.h.moments}.
 We consider~$V$ the solution to~\eqref{eq.V}. For any $0 \leq m \leq 2 n+1$,
 there exists a constant $C$ such that the following estimate holds:
 \begin{equation} \label{ineq.V.decay}
  \left| V(t, \cdot) \right|_{H^{s,m}(\R_+)}
  \leq C \left| \frac{\ln (2+t)}{2+t} \right|^{\f14 + \frac{2n+1}{2} - \frac{m}{2}}.
 \end{equation}
\end{lemma}

\begin{proof}
 Estimate~\eqref{ineq.V.decay} is straightforward up to time $T$ because its
 right-hand side is bounded from below for $t \in [0,T]$. Thus, we focus on
 large time estimates. We start by explicit computations in the frequency domain using Fourier transform. We consider the auxiliary system
 \begin{equation} \label{system.V.f}
 \left\{
 \begin{aligned}
   \partial_t f - \partial_{zz} f 
   & = (h(t) - \dot{h}(t)) \cdot \textrm{sgn}(z) e^{-|z|}, 
   && t \geq 0, \quad z \in \R, \\
   f(0,z) 
   & = 0, 
   && t = 0, \quad z \in \R. \\
 \end{aligned}
 \right.
 \end{equation}
 Since the source term in~\eqref{system.V.f} is odd, its unique solution $f$ satisfies $f(t,0) = 0$ for all $t \in \R_+$. Hence, thanks to the uniqueness property for the heat equation on the half-line, there holds $V(t,z) = f(t,z) + h(t) e^{-z}$ for $t,z \geq 0$ because both sides of this equality solve the same heat equation. Therefore, proving estimates on $f$ will provide estimates on $V$. After Fourier transform and solving the ODE, we obtain the formula:
 \begin{equation} \label{eq.V.f.hat}
  \hat{f}(t,\zeta) := \int_\R f(t,z) e^{-i \zeta z} \dz
  = -\frac{2i\zeta}{1+\zeta^2} \int_0^t e^{-(t-s)\zeta^2} \left(h(s)-\dot{h}(s)\right)\ds.
 \end{equation}
 Since $h$ vanishes after $T$ (see~\eqref{hyp.h.support}),
 the behavior of $f$ (and thus $V$) after time $T$ is entirely determined by
 the ``initial'' data $f_T(z) := f(T,z)$. Thanks to~\cite[Lemma 6]{2016arXiv161208087C},
 to establish~\eqref{ineq.V.decay}, it suffices to check that, for $0 \leq j \leq 2n$,
 \begin{equation} \label{moments.ft}
  \partial_\zeta^j \hat{f}_T(0) = 0.
 \end{equation}
 Thanks to~\eqref{eq.V.f.hat} and to the Leibniz rule, for $j \in \N$, one has:
 \begin{equation} \label{eq.dj.hatf}
  \partial_\zeta^j \hat{f}_T(\zeta)
  = - i \sum_{k=0}^j \binom{j}{k} \partial_\zeta^{j-k} \left\{ \frac{2\zeta}{1+\zeta^2} \right\}
  \int_0^T \left(h(t)-\dot{h}(t)\right) \partial_\zeta^k \left\{ e^{-(T-t)\zeta^2}\right\}
   \dt.
 \end{equation}
 First, since $\zeta \mapsto 2\zeta/(1+\zeta^2)$ is an odd function, only its 
 odd derivatives don't vanish at zero. Second, thanks to the Arbogast rule
 for the iterated differentiation of composite functions (also 
 known as Fa\`a di Bruno's formula), one has:
 \begin{equation} \label{eq.dk.dexpttzeta}
  \partial_\zeta^k \left\{ e^{-(T-t)\zeta^2}\right\}
  = \sum_{m_1 + 2 m_2 = k} \frac{k!}{m_1! m_2!} 
  \left( \frac{-2\zeta(T-t)}{1!} \right)^{m_1}
  \left( \frac{-2(T-t)}{2!} \right)^{m_2}
  e^{-(T-t)\zeta^2}.
 \end{equation}
 Hence, this derivative is non null at zero only if $k$ is even, say $k = 2k'$
 and the only non-vanishing term in the right-hand side of~\eqref{eq.dk.dexpttzeta}
 is the one corresponding to $(m_1,m_2) = (0,k')$ and is proportional to
 $(T-t)^{k'}$. From~\eqref{eq.dj.hatf} and~\eqref{eq.dk.dexpttzeta} we deduce
 that $\partial_\zeta^j \hat{f}_T(0)$ is a linear combination of the moments
 \begin{equation} \label{hh.moment.k'}
  \int_0^T \left(h(t)-\dot{h}(t)\right) (T-t)^{k'} \dt,
 \end{equation}
 where $0 \leq 2 k' \leq j - 1$. Thanks to~\eqref{hyp.h.support}
 and~\eqref{hyp.h.moments}, the integrals~\eqref{hh.moment.k'} vanish
 for $0 \leq k' < n$. So~\eqref{moments.ft} holds for $j \leq 2n-1$.
 Last,~\eqref{moments.ft} also holds for $j = 2n$ because, when $j$ is even,
 all the terms vanish. Indeed, in~\eqref{eq.dj.hatf}, either $k$ is odd
 or $j-k=2n-k$ is even. This concludes the proof of the lemma.
\end{proof}

\subsection{Fast variable scaling and Lebesgue norms}

Let us prove the following lemma, which is a simpler version 
of~\cite[Lemma 3, page 150]{MR2754340}.

\begin{lemma} \label{lemma.scale}
 Let $\gamma \in C^0([-1,1])$ with $\gamma \equiv 0$ on $\left(-\f13,\f13\right)$. For  $\mathcal{V} \in L^2(\R_+)$ and $\varepsilon > 0$:
\begin{equation} \label{eq.lemma.scale}
 \| \gamma \eval{\mathcal{V}} \|_{L^2(-1,1)} 
 \leq 
 2 \varepsilon^{\frac14} 
 \| \gamma \|_\infty \| \mathcal{V} \|_{L^2(\R_+)}.
\end{equation}
\end{lemma}

\begin{proof}
 For $-1 \leq y \leq -\f14$, we assumed $\varphi' = 1$. Thus, $\varphi(y) = 1 + y$. Recalling the fast variable notation \eqref{eval} and performing an affine change of variables gives
\begin{equation}
 \int_{-1}^{-\f13} \gamma^2(y) \mathcal{V}^2
 \left(\frac{\varphi(y)}{\sqrt{\varepsilon}}\right) \dd y
 = 
 \sqrt{\varepsilon} \int_{0}^{\frac{2}{3\sqrt{\varepsilon}}}
 \gamma^2(\sqrt{\varepsilon}z - 1) \mathcal{V}^2(z) \dd z
 \leq \sqrt{\varepsilon} \|\gamma\|^2_\infty \|\mathcal{V}\|^2_{L^2(\R_+)}.
\end{equation}
Proceeding likewise for $\f13 \leq y \leq 1$ and  bounding $\sqrt{2}$ by $2$ yields~\eqref{eq.lemma.scale},
\end{proof}

\subsection{Estimates for the technical profile}

\begin{lemma} \label{lemma.W}
 Assume that \eqref{hyp.h.moments} holds for some $n \geq 3$. There exists $C_W$ such that, for every $\varepsilon \in (0,1)$, the solution $W^\e$ to \eqref{eq.W} satisfies, for every $t\geq0$,
 \begin{equation} \label{estimate.W}
  \| W^\e(t) \|_{L^\infty(-1,1)} + \| \partial_y W^\e(t) \|_{L^\infty(-1,1)} 
  \leq \e^{-\frac34} C_W.
 \end{equation}
\end{lemma}

\begin{proof}
 Differentiating \eqref{eq.W} with respect to time, multiplying by $\partial_t W^\e$ and integrating by parts, we obtain the energy estimate
 \begin{equation}
  \| \partial_t W^\e \|_{L^\infty(\R_+;L^2(-1,1))}
  \leq 
  2 \| \partial_t f_W^\e \|_{L^1(\R_+;L^2(-1,1))}.
 \end{equation}
 Plugging this estimate in the equation \eqref{eq.W} yields
 \begin{equation} \label{yyW}
  \| \partial_{yy} W^\e \|_{L^\infty(\R_+;L^2(-1,1))}
  \leq 
  \frac{1}{\e} \left(
  \| f_W^\e \|_{L^\infty(\R_+;L^2(-1,1))} +
  2 \| \partial_t f_W^\e \|_{L^1(\R_+;L^2(-1,1))}
  \right).
 \end{equation}
 Thanks to estimate \eqref{eq.lemma.scale} from \cref{lemma.scale} applied to the definition \eqref{eq.FW} of $f_W^\e$, we obtain, for $t\geq 0$,
 \begin{equation} \label{fw1}
  \begin{split}
  \| f^\e_W(t) \|_{L^2(-1,1)} 
  & \leq
  2 \e^{\frac14} \| \chi'' \varphi^{-2} \|_{\infty} \| z^2 V(t,z) \|_{L^2(\R_+)}
  \\
  & \quad \quad + 2 \e^{\frac14} \| 2 \chi' \varphi' \varphi^{-3} \|_{\infty} \| z^3 \partial_z V(t,z) \|_{L^2(\R_+)}
  \\
  & \leq C \e^{\frac14} \| V(t) \|_{H^{1,3}(\R_+)},
  \end{split}
 \end{equation}
 where $C$ is a finite constant because, by construction, $\chi'$ and $\chi''$ vanish for $|y| \geq \frac23$, so that the division by $\varphi$ which vanishes at $y = \pm 1$ is not singular. Proceeding similarly and using the equation \eqref{eq.V} on $V$, we obtain
 \begin{equation} \label{fw2}
  \| \partial_t f^\e_W(t) \|_{L^2(-1,1)} 
  \leq C \e^{\frac14} \| V(t) \|_{H^{3,3}(\R_+)}.
 \end{equation}
 Combining \eqref{fw1} with \cref{Lemma:V} applied to $m = 3$, $n = 3$, $s = 1$, we obtain
 \begin{equation} \label{fw1.bis}
   \| f^\e_W(t) \|_{L^2(-1,1)} \leq C \e^{\frac14} \left| \frac{\ln (2+t)}{2+t} \right|^{\frac94}
   \leq C \e^{\frac14}
 \end{equation}
 Combining \eqref{fw2} with \cref{Lemma:V} applied to $m = 3$, $n = 3$, $s = 3$, we obtain
 \begin{equation} \label{fw2.bis}
  \| \partial_t f^\e_W \|_{L^1(\R_+;L^2(-1,1))} \leq C \e^{\frac14} \int_0^{+\infty} \left| \frac{\ln (2+t)}{2+t} \right|^{\frac94} \dd t 
  \leq 2 C \e^{\frac14}.
 \end{equation}
 Eventually, plugging \eqref{fw1.bis} and \eqref{fw2.bis} into \eqref{yyW} proves \eqref{estimate.W} thanks to the boundary conditions $W^\e(t,\pm 1) = 0$ and the Poincar\'e-Wirtinger inequality for $\partial_y W^\e$.
\end{proof}

\subsection{Proof of the decay of approximate trajectories}

We prove \cref{thm:uapp-estimate}. Recalling the definition \eqref{eq:uapp} of $\uapp$, we estimate the size of each term at the time $T/\e^\kappa$. 
\begin{itemize}
 \item Thanks to \eqref{hyp.h.support} and \eqref{eq:u0p0}, $u^0(T/\e^\kappa) = 0$.
 
 \item Thanks to \eqref{eq.lemma.scale} from \cref{lemma.scale} and \eqref{ineq.V.decay} from \cref{Lemma:V}, there holds
 \begin{equation}
  \begin{split}
   \| \chi \eval{v^0(T/\e^\kappa)} \|_{L^2_y}
  & \leq 2 \e^{\frac{1}{4}} \|\chi\|_\infty \| V(T/\e^\kappa) \|_{L^2(\R_+)}
  \\ & \leq 2 \e^{\frac{1}{4}} C \left| \frac{\ln (2+T/\e^\kappa)}{2+T/\e^\kappa}\right|^{\frac{3}{4}+n}
  \\ & \leq \tilde{C} \e^{1+\kappa(n-\frac{3}{4}(\frac{1}{\kappa}-1))} |\ln\e|^{n+\frac{3}{4}},
  \end{split}
 \end{equation}
 for some constant $\tilde{C}$ > 0.
 
 \item Thanks to \eqref{eq:magic-u1}, 
 \begin{equation} \label{u1_survivor}
  u^1(T/\e^\kappa) = \nabla^\perp[(1-\chi_\delta)\psi_b].
 \end{equation}
 Moreover, since $u_b$ satisfies \eqref{ub.int}, $u_b \in L^1(\band)$. In particular, since $u_b$ is divergence-free, this implies that, for all $x \in \R$,
 \begin{equation}
  \int_{-1}^{+1} u_b(x,y) \cdot \ex \dd y = 0, 
 \end{equation}
 so that $\psi_b$, which was defined as \eqref{psi_b} can equivalently be written as
 \begin{equation}
  \psi_b(x,y) = \int_{y}^1 u_b \cdot \ex \dd y.
 \end{equation}
 Thanks to \eqref{u1_survivor}, this implies that there exists a constant $C_{\delta} > 0$ which only depends on the norm of $\chi_\delta$ in $H^{1}(-1,1)$ such that
 \begin{equation}
  \e \| u^1(T/\e^\kappa) \|_{L^2(\band)} \leq \e C_\delta \| {u_b}_{\rvert \{|y| \geq 1-2\delta\}} \|_{L^2(\band)}.
 \end{equation}
 
 \item Thanks to estimate \eqref{estimate.W} from \cref{lemma.W}, 
 \begin{equation}
  \e^2 \| w^\e (T/\e^\kappa) \|_{L^\infty_y} \leq \e^{1+\frac{1}{4}} C_W.
 \end{equation}
 
\end{itemize}
Gathering these estimates concludes the proof of estimate \eqref{eq:uapp-small} of \cref{thm:uapp-estimate}.


\section{Estimates on the remainder}
\label{sec:r}

This section is devoted to the proof of \cref{thm:r-estimate}. An important difficulty to obtain some uniform energy estimates of $r^\varepsilon$ from system \eqref{eq.nse.r} is that the term $(r^\e \cdot \nabla) \uapp$ contains a term with a factor   $1/\sqrt{\varepsilon}$ due to the fast variation of the boundary layer term in the normal variable (see the expansion \eqref{eq:uapp} of $\uapp$). To deal with this difficulty we use a reformulation of this term where the singular factor is traded against a loss of derivative on $r^\varepsilon$ in the tangential direction $x$ (see \cref{ampli}). Then, we establish a long-time nonlinear Cauchy-Kovalevskaya estimate (see \cref{sec-analytic}) thanks to some tools from Littlewood-Paley theory which are recalled in \cref{LP}.

\begin{remark}
 The well-posedness of the Prandtl equations as well as the convergence of the Navier-Stokes equations to the Prandtl equations in the analytic setting dates back to \cite{MR1617542,MR1617538,MR2049030}. The seminal results of Caflisch and Sammartino require analyticity in both spatial directions, and only imply well-posedness of the Prandtl equations on a small time interval. Analytic techniques have been later used in \cite{MR3461362,MR3464051} to obtain large-time well-posedness for Prandtl equations by requiring analyticity only in the tangential direction.
\end{remark}


\subsection{Singular amplification to loss of derivative}
\label{ampli}

On the one hand, we use the expansion \eqref{eq:uapp} of $\uapp$ to expand
\begin{equation} \label{chaipa1}
 \begin{split}
  \left( r^\varepsilon \cdot \nabla \right) \uapp 
  = \frac{1}{\sqrt{\varepsilon}} \varphi' \chi r^\varepsilon_2 \eval{\partial_z  v^0}
  + r^\varepsilon_2 (\chi' \eval{v^0}  {+ \e^2 \partial_y w^\e}) 
  + \e \left( r^\varepsilon \cdot \nabla \right) u^1.
 \end{split}
\end{equation}
 Let  $M$  be the  operator which associates with any function $a \in L^2(-1,1)$, the function $M [a] $ defined for $y$ in $(-1,1)$ by 
\begin{equation}
  \label{eq:def-M}
(M [a])(y) := - \chi(y) \int_0^1 a (\pm 1 \mp s (1 \mp y)) \dd s,
 \end{equation}
where the signs are chosen depending on whether $\pm y \geq 0$. Using the null boundary condition and the divergence-free condition in \eqref{eq.nse.r} and the fact that $|\varphi'|=1$ where $\chi \neq 0$, we obtain that the first term in the right-hand side of \eqref{chaipa1} can be recast as 
\begin{equation} \label{recast1}
 \frac{1}{\sqrt{\varepsilon}} \varphi' \chi r^\varepsilon_2 \eval{\partial_z  v^0}
 = 
 (M [\partial_x r^\varepsilon_1]) \eval{z \partial_z  v^0}.
\end{equation}
On the other hand we decompose the term $( \uapp \cdot \nabla) r^\varepsilon$ of \eqref{eq.nse.r}, thanks to \eqref{eq:uapp}, into 
\begin{equation} \label{recast2}
 \left( \uapp \cdot \nabla \right) r^\varepsilon = (h - \chi \eval{V} {+ \e^2 W^\e}) \partial_x r^\varepsilon +  \varepsilon \left(  u^1  \cdot \nabla \right) r^\varepsilon.
\end{equation}
Thus, using \eqref{recast1} and \eqref{recast2}, the system \eqref{eq.nse.r} now reads
\begin{equation} \label{eq.nse.r.bis}
 \left\{
 \begin{aligned}
   \partial_t r^\varepsilon +  (h-\chi \eval{V} {+ \e^2 W^\e}) \partial_x  r^\varepsilon   - \varepsilon \Delta r^\varepsilon  + \nabla \pi^\varepsilon  & =  f^\varepsilon_{r}
   &&\quad \textrm{in } (0, T/ \varepsilon^\kappa) \times \band, \\
   \diverg r^\varepsilon & = 0
   &&\quad \textrm{in } (0, T/ \varepsilon^\kappa) \times \band, \\
   r^\varepsilon & = 0 
   &&\quad \textrm{on } (0, T/ \varepsilon^\kappa) \times \partial\band, \\
   r^\varepsilon \rvert_{t = 0} & = 0
   &&\quad \textrm{on } \band,
 \end{aligned}
 \right.
\end{equation}
where we introduce
\begin{equation}
 \label{sourcet}
 \begin{split}
  - f^\varepsilon_{r} 
  := 
  f^\varepsilon_{\mathrm{app}} 
  & + ( M [\partial_xr^\varepsilon_1 ]) \eval{z \partial_z  v^0} 
  + r^\varepsilon_2 (\chi' \eval{v^0} {+ \e^2 \partial_y w^\e})
  \\
  &  
  + \varepsilon \left( r^\varepsilon \cdot \nabla \right)  u^1 
  + \varepsilon \left(  u^1  \cdot \nabla \right) r^\varepsilon
  + \varepsilon \left(r^\varepsilon  \cdot \nabla \right) r^\varepsilon 
  .
 \end{split}
\end{equation}


\subsection{A few tools from Littlewood-Paley theory}
\label{LP}

To perform analytic estimates, we use Fourier analysis and Littlewood-Paley decomposition. We refer to \cite[Chapter 2]{MR2768550} for a detailed course on Littlewood-Paley theory. Although all the functions we consider in this section are defined on the band $\band = \R_x \times [-1,1]_y$, we only perform Fourier analysis and Littlewood-Paley decomposition in the tangential direction $x \in \R_x$. When a confusion is possible, we will use the subscripts $x$ or $y$ to stress the variable involved in the functional spaces. 

\paragraph{Dyadic partition of unity.} 

We recall that, for $a\in L^2(\band)$, we defined its Fourier transform $\cF a$ in the tangential direction as \eqref{eq:FOURIER}. We fix $\chilp, \philp \in C^\infty(\R,[0,1])$ such that
\begin{gather}
 \label{supp.philp}
 \supp \philp \subset \left \{\tau \in \R; \enskip \frac34 \leq |\tau| \leq \frac83 \right\}, 
 \\
 \label{supp.chilp}
 \supp \chilp \subset \left\{\tau \in \R; \enskip |\tau| \leq
\frac43 \right\},
 \\ 
 \label{philp1}
 \forall \tau \in \R^*, \quad 
 \sum_{j\in\Z} \philp(2^{-j}\tau)=1,
 \\
 \label{philp2}
 \forall \tau \in \R, \quad 
 \chilp(\tau)+ \sum_{j \in \N}\philp(2^{-j}\tau)=1,
 \\
 \label{philp3}
 \forall \tau \in \R^*, \quad
 \frac{1}{2} \leq \sum_{j\in\Z} \philp^2(2^{-j}\tau) \leq 1,
\end{gather}
The existence of such a dyadic partition of unity is proved in \cite[Proposition 2.10]{MR2768550}. For $k \in \Z$, we introduce the Fourier multipliers $\D_k$ and $\Low_k$ by defining, for any $a \in L^2(\band)$,
\begin{align}
 \label{def.dk}
 \D_k a & := \cF^{-1}\left(\philp(2^{-k}\xi)\mathcal{F}a(\xi,y)\right),
 \\
 \Low_k a & := \cF^{-1}\left(\chilp(2^{-k}\xi)\mathcal{F}a(\xi,y)\right).
\end{align}
The operators $\D_k$ and $\Low_k$ are with respect to the horizontal variable only. For $a \in L^2(\band)$, one has, thanks to \eqref{philp1} and \eqref{philp2},
\begin{equation} \label{SLOW}
 \Low_k a = \sum_{j\leq k-1} \D_j a.
\end{equation}

\paragraph{Homogeneous Besov spaces.} 

For $a \in L^2(\band)$, we will use for $s= 0$ and $s= \frac12$ the following quantity corresponding to a homogeneous Besov norm
\begin{equation} \label{eq:besov}
 \|a\|_{\besov{s}} := \sum_{k\in \Z} 2^{ks}\|\D_k
a\|_{L^2(\band)}.
\end{equation}
Since we will use such norms for functions whose Fourier transforms in $x$ are compactly supported, we do not provide more details on the definition of the corresponding functional spaces, referring for more to  \cite{MR2768550}.

\paragraph{Classical estimates.}

We recall the following classical estimates, for which we track the constants. First, we will use the following Bernstein type lemma from \cite[Lemma 2.1]{MR2768550}.

\begin{lemma} \label{lem:Bern}
 There exists a universal constant $\cb \geq 2$ such that the following properties hold. Let $1 \leq p \leq q \leq + \infty$, $\alpha \in \{0,1\}$, $k \in \Z$ and $a \in L^2(\band)$. 
 \begin{itemize}
  \item If the support of $\cF a$ is included in $\{ (\xi,y); \enskip 2^{-k} |\xi| \leq 100 \}$, then
 \begin{equation} \label{eq:bern-ball}
  \|\partial_{x}^\alpha a\|_{L^{q}_{x}(L^2_y)} \leq \cb
  2^{k\left(\alpha+\left(\frac{1}{p}-\frac{1}{q}\right)\right)}
  \|a\|_{L^{p}_{x}(L^2_y)}.
 \end{equation}
  \item If the support of $\cF a$ is included in $\{ (\xi,y); \enskip \f1{100} \leq 2^{-k} |\xi| \leq 100 \}$, then
 \begin{equation} \label{eq:bern-ring}
  \|a\|_{L^{p}_{x}(L^2_y)} \leq \cb
    2^{-k\alpha} \|\partial_{x}^\alpha a\|_{L^{p}_{x}(L^2_y)}.
 \end{equation}
 \end{itemize}
\end{lemma}

\begin{lemma} \label{lem:gns}
 Let $a \in H^1_0([-1,1]_y)$. Then
 \begin{equation} \label{eq:gns}
  \| a \|_{L^\infty_y} \leq \| a \|_{L^2_y}^{\f12} \| \partial_y a \|_{L^2_y}^{\f12}.
 \end{equation}
\end{lemma}

\begin{proof}
 This is a classical Gagliardo-Nirenberg interpolation inequality (see \cite{MR0109940}). The fact that \eqref{eq:gns} holds with a unit constant for this particular choice of exponents is proved for example in \cite[Corollary 5.12]{MOROSI2017} (which in fact yields a constant $2^{-\f12}$). 
\end{proof}

As a consequence of \cref{lem:Bern}, we have the following embedding. Indeed this is the main motivation for considering the $\ell^{1}$ norm rather than the $\ell^{2}$ norm in the definition of the homogeneous Besov norms $\besov{s}$.

\begin{lemma} \label{lem:embedding}
 Let $a \in H^1_0(\band)$. There holds,
 \begin{align} 
  \label{embed1}
  \sum_{k\in\Z} 2^{\frac{k}2} \|\D_k a\|_{L^2_x (L^\infty_y)} 
  & \leq \cb \|\na a\|_{\besov{0}}, \\
  \label{embed2}
  \sum_{k\in\Z} \|\D_k a \|_{L^\infty(\band)} 
  & \leq \cb^2 \|\na a\|_{\besov{0}}.
 \end{align}
\end{lemma}
\begin{proof}
 Let $a \in H^1_0(\band)$. Hence, for almost every $x \in \R_x$, $a(x,\cdot) \in H^1_0([-1,1]_y)$ and we can apply \cref{lem:gns}. Using \eqref{eq:gns}, Cauchy-Schwarz then \eqref{eq:bern-ring} yields
 \begin{equation} \label{eq:jus1}
  \begin{split}
    2^{\frac{k}2}\|\D_k a\|_{L^2_x (L^\infty_y)}
    & \leq 2^{\frac{k}2}\|\D_k a\|_{L^2}^{\f12}\|\D_k \partial_y a\|_{L^2}^{\f12}
    \\
    & \leq \cb^{\f12} \|\D_k \p_x a\|_{L^2}^{\f12}\|\D_k \partial_y a\|_{L^2}^{\f12}
    \\
    & \leq \cb^{\f12} \|\D_k \nabla a\|_{L^2},
  \end{split}
 \end{equation}
 Hence, since $\cb \geq 1$, \eqref{eq:jus1} proves \eqref{embed1} by the definition \eqref{eq:besov} of the norm $\besov{0}$. Moreover, thanks to \eqref{eq:bern-ball},
 \begin{equation} \label{eq:jus2} 
  \|\D_k a\|_{L^\infty_x(L^\infty_y)} 
  \leq \cb 2^{\frac{k}2}\|\D_k a\|_{L^2_x (L^\infty_y)}.
 \end{equation}
 Gathering \eqref{embed1} and \eqref{eq:jus2} proves \eqref{embed2}.
\end{proof}

\begin{lemma} \label{lem:tg}
 Let $a \in H^1_0(\band)$ such that $\diverg a = 0$. For each $k \in \Z$,
\begin{equation} \label{eq:tg}
\|\D_k a_2 \|_{L^2_x(L^\infty_y)} \leq \cb 2^{\f{k}2} \|\D_k a\|_{L^2(\band)}, 
 \end{equation}
 \end{lemma}
\begin{proof}
 Let $a \in H^1_0(\band)$. Hence, for almost every $x \in \R_x$, $a(x,\cdot) \in H^1_0([-1,1]_y)$ and we can apply \cref{lem:gns}. Using \eqref{eq:gns} and Cauchy-Schwarz, we obtain
\begin{equation} \label{eq:tq1}
\|\D_k a_2\|_{L^2_x(L^\infty_y)} \leq \|\D_k a_2\|_{L^2}^{\f12}\|\D_k\p_y a_2\|_{L^2}^{\f12} .
\end{equation}
 Then, using that $\diverg a = 0$ and \cref{lem:Bern}, we observe that 
\begin{equation} \label{eq:tq2}
 \|\D_k\p_y a_2\|_{L^2} =  
 \|\D_k\p_x a_1 \|_{L^2} \leq \cb  2^{k}   \|\D_k a_1\|_{L^2}.
\end{equation}
 Gathering \eqref{eq:tq1} and \eqref{eq:tq2} proves \eqref{eq:tg} since $\cb \geq 1$.
\end{proof}

\paragraph{Paraproduct decomposition.}
We shall use the Bony's decomposition (see \cite{MR631751}) for the horizontal variable: 
\begin{equation} \label{Bony} 
fg=\TT_f g+\TT_{g}f+\RR(f,g), 
\end{equation} 
where 
\begin{align} 
 \TT_f g & := \sum_k \Low_{k-1}f\D_k g,
 \\
 \RR(f,g) & := \sum_k{\D}_kf\widetilde{\D}_{k}g
 \\
 \label{eq:wt-delta}
 \widetilde{\D}_k g & := 
 \sum_{|k-k'|\le 1}\D_{k'}g. 
\end{align} 
Thanks to the support properties \eqref{supp.philp} of $\philp$ and \eqref{supp.chilp} of $\chilp$, the following lemma holds.

\begin{lemma} \label{support}
For any $f$, $g$ and $h$ in $L^2(\band)$, 
\begin{align}
 \label{supportT}
 \langle  \TT_{f} g  ,\D_k h \rangle
 & = 
 \sum_{k'\in \Z / \ |k'-k|\leq 4} \, 
  \langle (\Low_{k'-1}  f)(\D_{k'} g) ,\D_k h \rangle,
  \\ \label{supportR}
  \langle  \RR(f,g)  ,\D_k h \rangle
 & = 
 \sum_{k' \in \Z/ \ k' \geq k-3} \, 
  \langle  ( \D_{k'}  f)(\widetilde{\D}_{k'} g)  ,\D_k h \rangle.
 \end{align}
\end{lemma}

\paragraph{Analyticity by Fourier multipliers.}
Let $|\partial_x|$ denote the Fourier multiplier with symbol~$|\xi|$. We associate with any positive $C^1$ function of time $\rho$, the operator $\opphi$ mapping any reasonable function $f(t,x,y)$ (say such that $f \in L^1_{\mathrm{loc}}(L^2_N(\band))$, for some $N\in\N$), to
\begin{equation} \label{eq2.4}
 (\opphi f )(t,x,y) := \cF^{-1}\bigl(e^{\rho(t) |\xi| }\cF f(t,\xi,y)\bigr)(x). 
\end{equation} 
Recall that $\mathcal{F}$ denotes the Fourier transform with respect to the tangential variable $x$, see~\eqref{eq:FOURIER}.
The function  $\rho$  describes the evolution of the radius of analyticity of the considered function. Below we establish a long-time Cauchy-Kovalevskaya estimate, for which the function  $\rho$  decays in time but not linearly. 

\paragraph{Product estimates for analytic functions.}
For $a \in L^2(\band)$, we introduce the notation
\begin{equation} \label{eq:plus}
 a^+ := \cF^{-1} |\mathcal{F}a|.
\end{equation}

\begin{lemma} \label{vandamme}
Let $N\in\N^*$ and $a, b, c \in L^2_N(\band)$. There holds
\begin{align} 
 \label{eq:vandamme0}
 \| a^+ \|_{L^2} & = \| a \|_{L^2}, 
 \\
 \label{eq:vandamme_fou}
 \left| \langle \opphi \pn (ab), c \rangle \right| 
 & \leq \left| \langle (\opphi a^+) (\opphi b^+), c^+ \rangle \right|.
 \end{align}
\end{lemma}
\begin{proof}
 Equality \eqref{eq:vandamme0} is an immediate consequence of the definition \eqref{eq:plus} and Plancherel's theorem. Moreover, by Plancherel's theorem, the normalization \eqref{eq:FOURIER}, the triangle inequality and Plancherel's theorem once more, we have that
 \begin{equation}
  \begin{split}
  \left| \langle \opphi \pn (ab), c \rangle \right| 
  & = \frac{1}{2\pi} \left| \int_y \int_{|\xi|\leq N} \cF c(\xi) e^{\rho |\xi|} 
  \int_\eta \cF a (\xi - \eta ) \cF b(\eta) \dd \eta \dd \xi \dd y
  \right| 
  \\
  & \leq 
  \frac{1}{2\pi} \int_y \int_{\xi\in\R} |\cF c(\xi)|  
  \int_\eta e^{\rho |\xi-\eta|} |\cF a (\xi - \eta )| e^{\rho |\eta|} |\cF b(\eta)| \dd \eta \dd \xi \dd y
  \\
  & =  \langle (\opphi a^+) (\opphi b^+), c^+ \rangle.
  \end{split}
 \end{equation}
 This scalar product is positive and this concludes the proof of \eqref{eq:vandamme_fou}.
\end{proof}
%

\subsection{Long-time weakly nonlinear Cauchy-Kovalevskaya estimate}
\label{sec-analytic}

In this paragraph, we explain how we will prove a long-time weakly nonlinear Cauchy-Kovalevskaya estimate on the remainder. We start by defining quantities that will enable us to define the expected profile of analyticity $\rho(t)$. Then, we close the estimate relying on a Grönwall-type argument. In the following paragraphs, we will prove the required estimates.

\begin{remark}
 The idea of closing an estimate on a nonlinear function of the solution to control the loss of analyticity dates back to Chemin in \cite{MR2145938}. It was later used in the context of anisotropic Navier-Stokes equations in \cite{MR2776367} and, more recently, for Prandtl equations in~\cite{MR3464051}, using only analyticity in the tangential direction.
\end{remark}

\subsubsection{Friedrichs' regularization scheme}

In order for our manipulations to make sense, we will restrict \eqref{eq.nse.r.bis} to a bounded range of frequencies. Then, we establish estimates which are independent on the considered range and we pass to the limit. This process was introduced by Friedrichs in \cite{MR0009701} (see also \cite{MR2155019} for a recent example of the passage to the limit). Let $N\in \N$. Instead of \eqref{eq.nse.r.bis}, we consider the modified equation
\begin{equation} \label{eq.nse.r.ter}
 \left\{
 \begin{aligned}
   \partial_t r^\varepsilon_N +  (h-\chi \eval{V} {+ \e^2 W^\e}) \partial_x  r^\varepsilon_N   - \varepsilon \Delta r^\varepsilon_N  + \nabla \pi^\varepsilon_N  & =  f^\varepsilon_N
   &&\quad \textrm{in } (0, T/ \varepsilon^\kappa) \times \band, \\
   \diverg r^\varepsilon_N & = 0
   &&\quad \textrm{in } (0, T/ \varepsilon^\kappa) \times \band, \\
   r^\varepsilon_N & = 0 
   &&\quad \textrm{on } (0, T/ \varepsilon^\kappa) \times \partial\band, \\
   r^\varepsilon_N \rvert_{t = 0} & = 0
   &&\quad \textrm{on } \band,
 \end{aligned}
 \right.
\end{equation}
where we introduce
\begin{equation}
 \label{source.ter}
 \begin{split}
  - f^\varepsilon_N 
  := \pn 
  f^\varepsilon_{\mathrm{app}} 
  & + ( M [\partial_x r^\varepsilon_{N,1} ]) \eval{z \partial_z  v^0} 
  + r^\varepsilon_{N,2} (\chi' \eval{v^0} {+ \e^2 \partial_y w^\e})
  \\
  &  
  + \varepsilon \pn \left( r^\varepsilon_N \cdot \nabla \right)  u^1 
  + \varepsilon \pn \left(  u^1  \cdot \nabla \right) r^\varepsilon_N
  + \varepsilon \pn \left( r^\varepsilon_N  \cdot \nabla \right) r^\varepsilon_N
  .
 \end{split}
\end{equation}
In the sequel, to lighten the notations, we will write $r$ instead of $r^\e_N$ and we will omit the projections $\pn$. It will be clear from our proof that we perform \emph{a priori} estimates which are independent of $N$. Therefore, using usual compactness arguments, our proof will also yield the same energy estimate for the initial equation \eqref{eq.nse.r.bis}. Since this argument is quite classical, we will only detail the \emph{a priori} estimates. Even though this regularization process is transparent in the proof, it is necessary to ensure that all the quantities are well defined. 

\subsubsection{Definition of the analyticity profile}

We start by defining the analyticity radius that we will require on the coefficients and the source terms of the equation for the remainder
\begin{equation}
 \rho_0 := 2 + 10^2 \cb \int_0^{+\infty} \| z \partial_z v^0(t,z) \|_{L^\infty(\R_+)} \dd t.
\end{equation}
Recalling the definition \eqref{def.Hsm} of the space $H^{2,2}(\R_+)$, one has, for $t \geq 0$,
\begin{equation}
 \| z\partial_z V(t,z) \|_{L^\infty(\R_+)} \leq 2 \| V(t) \|_{H^{2,2}(\R_+)}.
\end{equation}
Hence, since $n \geq 2$, thanks to the decay estimate \eqref{ineq.V.decay} from \cref{Lemma:V}, $\rho_0 < +\infty$.
Up to a normalization constant due to Bernstein-type estimates, this radius corresponds to the total amount of the loss of derivative that we expect. Then, we set, for $t \geq 0$,
\begin{align} 
 \label{defalpha}
 \alpha_\e(t) & := \e^\kappa + (1+t^2)^{-1},
 \\
 \label{ell1}
 \ell_1(t) & := \sum_{k \in \Z} \| \oprhozero \D_{k} \nabla u^1(t) \|_{L^2(\band)}.
\end{align}
These quantities will help us to control the (non singular but long-time) amplification terms in the evolution of the remainder. We set
\begin{equation} \label{lstar}
 \beta(t) := \int_0^t \Big( 5 \|  \chi' \eval{v^0} {+\e^2 \p_y w^\e}  \|_{L^\infty (\band)} 
  + 10^9 \cb^2 \e \ell_1^2
  + 10 \alpha_\e \Big).
\end{equation}

\begin{proposition}[Proof in \cref{sec:beta}] \label{thm:beta}
 If $u_b$ satisfies \eqref{ub.anal} for a constant $C_b > 0$ and $\rho_b > \rho_0$, there exists $\beta_\star > 0$ such that, for $\e,\kappa \in (0,1)$,
 \begin{equation}
  \sup_{t\in[0,T/\varepsilon^\kappa]} \beta(t) 
  = \beta(T/\varepsilon^\kappa) 
  \leq \beta_\star.
 \end{equation}
\end{proposition}

We consider the local solution $\rho_N(t)$ to the following nonlinear ODE:
\begin{equation} \label{RHO}
 \left\{
 \begin{aligned}
  \dot{\rho}_N(t) & = - 10^2 \cb \| z \partial_z v^0(t,z) \|_{L^\infty(\R_+)} 
  - 10^7 \cb^4 \e e^{\beta_\star} \|e^{\rho_N(t)|\partial_x|-\beta(t)} \na r^\e_N(t) \|_{\besov{0}} ,
  \\
  \rho_N(0) & = \rho_0 .
 \end{aligned}
 \right.
\end{equation}
Since, for almost every $t$, $r^\e_N(t) \in L^2_N(\band)$, the right-hand side is Lipschitz continuous with respect to $\rho_N$ (with constants that may depend on $N$). Hence, we can apply the Cauchy-Lipschitz theorem and consider the maximal solution of \eqref{RHO}. We set
\begin{equation}
 T^\ast_N := \sup~\left\{ t \in [0,T/\e^\kappa]; \enskip \rho_N(t) \geq 1 \right\}
\end{equation}
and consider for $t \leq T^\ast_N$, 
\begin{equation}
\label{niouvariable}
 \rphi := e^{\rho_N|\partial_x|-\beta} r^\varepsilon_N.
\end{equation}
In the sequel, we simply write $\rho$ instead of $\rho_N$ and $T^\ast$ instead of $T^\ast_N$ and we prove estimates which are uniform with respect to $N$.

\subsubsection{Grönwall-type energy estimate}

We start with deducing from \eqref{eq.nse.r.ter} that:
\begin{equation} \label{eq.nse.rPhi}
 \left\{
 \begin{aligned}
     \partial_t \rphi - \dot{\rho}  \vert \partial_x  \vert \rphi
 + \dot{\beta} \rphi 
  + (h-\chi \eval{V} {+\e^2 W^\e}) \partial_x  \rphi & && \\
 - \varepsilon \Delta \rphi   + \nabla  \opphibeta \pi   & =   \opphibeta f^\e_N
   &&\quad \textrm{on } (0, \frac{T}{\varepsilon^\kappa}) \times \band, \\
   \diverg \rphi  & = 0
   &&\quad \textrm{on } (0, \frac{T}{\varepsilon^\kappa}) \times \band, \\
    \rphi  & = 0 
   &&\quad \textrm{on } (0, \frac{T}{\varepsilon^\kappa}) \times \Gamma_\pm, \\
   \rphi  \rvert_{t = 0} & =0
   &&\quad \textrm{on } \band.
 \end{aligned}
 \right.
\end{equation}
We apply the dyadic operator $\D_k$ to \eqref{eq.nse.rPhi} and take the $L^2 (\band)$ inner product of the resulting equation with $\D_k\rphi$. We observe, by integration by parts, that the contributions due to the fourth and sixth terms vanish, so that 
\begin{equation}
 \begin{split}
 \f12\f{\dd}{\dd t}\|\D_k \rphi(t)\|_{L^2}^2
-  \dot{\rho} \,   \langle|\partial_x|\D_k\rphi,\D_k\rphi\rangle
& + \dot{\beta} \| \D_k \rphi(t) \|_{L^2}^2
+ \e \| \nabla  \D_k\rphi \|_{L^2}^2
\\
& = \langle\D_k  \opphibeta f^\e_N, \D_k\rphi\rangle.
 \end{split}
\end{equation} 
Above and below we simply denote by $L^2 $ the space $L^2 (\band)$. 
%
%
Using the definition \eqref{def.dk} of $\D_k$ and the support property \eqref{supp.philp} of $\philp$, we know that
\begin{equation}
 \langle |\partial_x| \D_k \rphi, \D_k \rphi \rangle \geq \frac{1}{2} 2^{k} \| \D_k \rphi \|_{L^2}^2.
\end{equation}
Then, integrating over $[0,t]$, we obtain
\begin{equation} \label{eq3.2a}
\begin{split}
 \f12 \|\D_k \rphi(t)\|_{L^2}^2
 & +  \f12 2^k \int_0^t |\dot{\rho}| \ \|\D_k\rphi\|_{L^2}^2
 + \int_0^t \dot{\beta} \| \D_k \rphi \|_{L^2}^2
 +  \e \int_0^t \| \nabla \D_k \rphi \|_{L^2}^2 
 \\
 \leq &
 \int_0^t\bigl|\langle\D_k  \opphibeta f^\e_N ,\D_k\rphi\rangle\bigr|.
\end{split}
\end{equation}
We take the square roots and  sum the resulting inequalities for $k\in\Z$ to deduce that 
\begin{equation} \label{avecrac}
\begin{split}
 & \sum_{k \in\Z}\|\D_k \rphi(t)\|_{L^2}
 + \sum_{k \in\Z} 2^\frac{k}{2}  \left( \int_0^t |\dot{\rho}|\,  \|\D_k\rphi\|_{L^2}^2 \right)^\frac12
 + \sqrt{2} \sum_{k\in\Z} \left( \int_0^t \dot{\beta} \|\D_k \rphi \|_{L^2}^2 \right)^\frac{1}{2}
 \\
 & + \sqrt{2\e} \sum_{k \in\Z}  \left(  \int_0^t \|\D_k \nabla \rphi\|_{L^2}^2 \right)^{\f12}
 \leq 
2 \sqrt{2} \sum_{k \in\Z} \left( \int_0^t |\langle\D_k  \opphibeta f^\e_N ,\D_k\rphi\rangle|\right)^\frac12.
\end{split}
\end{equation}

\begin{proposition}[Proof in \cref{sec-proof-grostruc}] 
 \label{grotruc}
 For $t \in [0,T^\ast]$, there holds
 \begin{equation}
  \label{eq-grotruc}
 \begin{split}
 2\sqrt{2} \sum_{k \in\Z} & \left(\int_0^t\bigl|\bigl(\D_k  \opphibeta f^\e_N   , \D_k \rphi\rangle\bigr|\right)^\frac12 
 \leq 
 \sum_{k\in\Z} 2^{\frac{k}{2}} \left( \int_0^t |\dot{\rho}|\,  \|\D_{k'} \rphi \|^2_{L^2} \right)^\frac12 
 \\ 
 & \quad 
 + \sqrt{2} \sum_{k\in\Z} \left( \int_0^t \dot{\beta} \|\D_{k} \rphi \|^2_{L^2} \right)^{\f12} 
 + \f14 \sqrt{\e} \sum_{k\in\Z} \left(  \int_0^t \|\D_k \nabla \rphi\|_{L^2}^2\right)^{\f12}
 \\
 & \quad 
 + \sqrt{2} \sum_{k \in\Z} \left(\int_0^t \frac{1}{\alpha_\e} \|  \oprhozero \D_k f^\varepsilon_{\mathrm{app}}\|_{L^2}^2 \right)^\frac12
 .
 \end{split}
\end{equation} 
\end{proposition}
The proof of Proposition \ref{grotruc} is given in Section \ref{sec-proof-grostruc}. Let us admit Proposition \ref{grotruc} for the time being and see how to conclude the  proof of  \cref{thm:r-estimate}. Combining \eqref{avecrac}  and  \eqref{eq-grotruc} we deduce that 
\begin{equation}  \label{gronwall1}
 \begin{split}
  & \sum_{k \in\Z}\|\D_k\rphi(t)\|_{L^2}
  + \sqrt{\e} \sum_{k \in\Z}  \left(  \int_0^t \|\D_k \nabla \rphi\|_{L^2}^2 \right)^{\f12}
  \\
  & \leq
  \sqrt{2} \sum_{k \in\Z} \left(\int_0^t \frac{1}{\alpha_\e} \| \oprhozero \D_k f^\varepsilon_{\mathrm{app}}\|_{L^2}^2 \right)^\frac12.
 \end{split}
\end{equation}

\begin{proposition}[Proof in \cref{sec:cf}] \label{thm:cf}
 If $u_b$ satisfies \eqref{ub.anal} and \eqref{ub.int} for a constant $C_b > 0$ and $\rho_b > \rho_0$, there exists $C_f > 0$ such that, for  $\e,\kappa \in (0,1)$,
 \begin{equation}
  \sup_{t \in [0,T/\varepsilon^\kappa]} 
  \sum_{k \in\Z} \left(\int_0^t \frac{1}{\alpha_\e} \| \oprhozero \D_k f^\varepsilon_{\mathrm{app}}\|_{L^2}^2 \right)^\frac12
  \leq C_f (\e^{\frac{1}{4}} + \e^{1-\kappa}).
 \end{equation}
\end{proposition}
As long as $t \leq T^\ast$, \eqref{gronwall1} holds and, thanks to \cref{thm:cf}, we obtain
\begin{equation}  \label{gronwall2}
  \sum_{k \in\Z}\|\D_k\rphi(t)\|_{L^2}
  + \sqrt{\e} \sum_{k \in\Z}  \left(  \int_0^t \|\D_k \nabla \rphi\|_{L^2}^2 \right)^{\f12}
  \leq
  \sqrt{2} C_f (\e^{\frac{1}{4}} + \e^{1-\kappa}).
\end{equation}
Moreover,
\begin{equation}
 \int_0^t |\dot{\rho}|
 \leq 10^2 \cb \int_0^{+\infty} \| z\partial_z v^0(t,z) \|_{L^\infty(\R_+)}
 + 10^7 \cb^4 \e e^{\beta_\star} \int_0^t \| \nabla \rphi \|_{\besov{0}}
\end{equation}
and, for $t \leq T/\e^\kappa$,
\begin{equation}
 \begin{split}
  \e \int_0^t \| \nabla \rphi \|_{\besov{0}}
  & = \e \sum_{k\in\Z} \int_0^t \| \D_k \nabla \rphi \|_{L^2}
  \\
  & \leq \sqrt{\e t} \cdot \sqrt{\e} \sum_{k\in\Z} \left( \int_0^t \| \D_k \nabla \rphi \|_{L^2}^2 \right)^{\f12} 
  \\
  & \leq \sqrt{2T} C_f (\e^{\frac{1}{4}} + \e^{1-\kappa}).
 \end{split}
\end{equation}
Combining these estimates yields
\begin{equation}
 \rho(T/\varepsilon^\kappa) \geq 2 - 10^7 \cb^4 e^{\beta_\star} \sqrt{2T} C_f (\e^{\frac{1}{4}} + \e^{1-\kappa}).
\end{equation}
Thus, for $\varepsilon$ small enough, $\rho(T/\e^\kappa) \geq 1$ and thus $T^\ast_N = T/\varepsilon^\kappa$ and one has
\begin{equation}
 \| r^\e_N \|_{L^\infty(L^2(\band))} 
 + \sqrt{\e} \| \nabla r^\e_N \|_{L^2(L^2(\band))}
 \leq
 \sqrt{2} e^{\beta_\star} C_f (\e^{\frac{1}{4}} + \e^{1-\kappa}).
\end{equation}
This estimate being uniform with respect to $N$, one can pass to the limit (for fixed $\e$) towards $r^\e$, and then take $\e$ small enough to conclude the proof of  \cref{thm:r-estimate}.



\subsection{Proof of Proposition \ref{grotruc}}
\label{sec-proof-grostruc}

To prove Proposition \ref{grotruc} we estimate separately the terms corresponding to the different terms of the decomposition of the source term $f^\e_N$  in \eqref{source.ter}. Let us start with the term corresponding to a loss of derivative. 

\begin{lemma}\label{estif1}
 For $t \in [0,T^\ast]$, there holds
\begin{gather*}
 \sum_{k\in\Z} 
 \left( \int_0^t\bigl|\langle\D_k  \opphibeta  ( M [\partial_x r_1]) \eval{z \partial_z  v^0} ,\D_k \rphi\rangle\bigr| \right)^{\f12} 
 \leq \\ \quad   
 \sum_{k\in\Z} 2^{\f{k}2} \left(  \int_0^t 2\cb \| z \partial_z v^0 \|_{L^\infty_z}  \|\D_{k} \rphi \|^2_{L^2} \right)^{\f12} .
\end{gather*}
\end{lemma}
\begin{proof}
 Since $\eval{z \partial_z v^0}$ and the operator $M$ do not depend on the $x$ variable, 
 \begin{equation}
  \D_k \opphibeta  ( M [\partial_x r_1]) \eval{z \partial_z  v^0} =  (M [\D_k \partial_x \rphi_1 ]) \eval{z \partial_z  v^0}.
 \end{equation}
 Moreover, using the definition of $M$ in \eqref{eq:def-M}, Hardy's inequality, and the fact that $|\chi| \leq 1$, we get that, for any $a \in L^2_y(-1,1)$,
 \begin{equation}
  \| M[a] \|_{L^2_y} \leq 2 \| a \|_{L^2_y}.
 \end{equation}
 Hence, using \eqref{eq:bern-ball} from \cref{lem:Bern}, we obtain
 \begin{equation}
  \bigl|\bigl(\D_k  \opphibeta  ( M [\partial_x r_1]) \eval{z \partial_z  v^0}, \D_k \rphi\rangle\bigr|
  \leq 2^{k+1} \cb  \| z\partial_z v^0 \|_{L^\infty_z}  \| \D_k \rphi \|_{L^2}^2.
 \end{equation}
 The result follows by integration in time and summation over $k \in \Z$ of the square roots.
\end{proof}


\begin{lemma}\label{estif2}
 For $t \in [0,T^\ast]$, there holds
\begin{equation} 
 \begin{split}
 \sum_{k\in\Z} &
 \left(\int_0^t \bigl| \langle\D_k  \opphibeta   r_2  (\chi' \eval{v^0} {+\e^2 \p_y w^\e}) ,\D_k \rphi\rangle \bigr|\right)^{\f12} 
 \\
 & \leq  \sum_{k\in\Z} \left( \int_0^t \|  \chi' \eval{v^0} {+\e^2 \p_y w^\e}  \|_{L^\infty (\band)}  \|\D_{k} \rphi \|^2_{L^2} \right)^{\f12}
 .
 \end{split}
\end{equation}
\end{lemma}
\begin{proof}
Since $\chi' \eval{v^0} {+\e^2 \p_y w^\e}$ does not depend on $x$, 
\begin{equation}
 \D_k  \opphibeta ( r_2  (\chi' \eval{v^0} {+\e^2 \p_y w^\e})  ) =  (\chi' \eval{v^0} {+\e^2 \p_y w^\e}) \D_k  \rphi_2 ,
\end{equation} 
and therefore the result readily follows by the Cauchy-Schwarz inequality.
\end{proof}


\begin{lemma}[Proof in \cref{sec:estim-ampli}]
 \label{lem:estim-ampli}
 For $t \in [0,T^\ast]$, there holds
 \begin{equation} 
  \begin{split}
  \sum_{k\in\Z} & 
  \left( \int_0^t \bigl| \langle\D_k  \opphibeta \left[ 
  \left( r \cdot \nabla \right)  u^1 
  + \left( u^1 \cdot \nabla \right) r \right],
  \D_k \rphi\rangle \bigr|\right)^{\f12} 
  \\ &
  \leq 
  10^8 \cb^2  \sum_{k\in\Z} \left( \int_0^t \ell_1^2 \|\D_{k} \rphi \|^2_{L^2} \right)^{\f12}
  + 
 \frac{1}{20} \sum_{k\in\Z} \left(  \int_0^t \|\D_k \nabla \rphi\|_{L^2}^2 \right)^{\f12}
 .
  \end{split}
 \end{equation}
\end{lemma}


\begin{lemma}[Proof in \cref{trili}] \label{lem:trilinear}
 For $t \in [0,T^\ast]$, there holds
\begin{equation*} 
\sum_{k \in\Z} \left(\int_0^t\bigl|\langle\D_k  \opphibeta ( \left(r \cdot \nabla \right) r),\D_k \rphi\rangle\bigr|\right)^\frac12
\leq 600 \cb^2 \sum_{k'\in\Z} 2^{\f{k'}2}  \left(\int_0^t e^{\beta_\star} \|\na \rphi\|_{\besov{0}} \|\D_{k'} \rphi \|^2_{L^2} \right)^\frac12  .
\end{equation*}
\end{lemma}

\cref{grotruc} follows from the definition \eqref{source.ter} of $f^\e_N$, \cref{estif1}, \cref{estif2},  \cref{lem:estim-ampli} and \cref{lem:trilinear}. Observe in particular that the sum of the right hand sides of \cref{estif1} and of \cref{lem:trilinear} can be bounded by the first term in the right hand sides of the estimate in \cref{grotruc} thanks to \eqref{RHO}. On the other hand the  sum of the right hand sides of \cref{estif2} and \cref{lem:estim-ampli} can be bounded by the other terms in the right hand sides of the estimate in \cref{grotruc} thanks to \eqref{lstar}.


\subsection{Estimate of \texorpdfstring{$ (r\cdot\na) u^1 + (u^1 \cdot\na) r$}{the linear terms}. Proof of Lemma \ref{lem:estim-ampli}}
\label{sec:estim-ampli}

Due to $\diverg r=0,$ we get, integrating by parts, that
\begin{equation}\label{intparts}
\langle\D_k  \opphibeta   \left( r \cdot \nabla \right)  u^1 ,\D_k \rphi\rangle 
=-\sum_{i,j=1}^2\langle\D_k  \opphibeta   \left( r_i u^1_j \right),\D_k \p_i\rphi_j\rangle.
\end{equation}
Due to $\diverg u^1=0,$ we get, integrating by parts, that
\begin{equation} \label{intparts2}
 \langle\D_k  \opphibeta  \left( u^1 \cdot \nabla \right) r,\D_k \rphi\rangle=-\sum_{i,j=1}^2\langle\D_k  \opphibeta  \left( u^1_i r_j \right),\D_k \p_i \rphi_j\rangle.
\end{equation}
This yields a total of 8 scalar terms, which we estimate separately using the same method for each set of indexes. We explain the proof only for the terms in \eqref{intparts} (the terms of \eqref{intparts2} are handled similarly).
By Bony's decomposition \eqref{Bony}, we expand the products as
\begin{equation} \label{sexpistols}
 r_i u^1_j = \TT_{r_i} u^1_j + \TT_{  u^1_j}  r_i + \RR(r_i ,  u^1_j).
\end{equation}
We explain the three estimates for each term in the right-hand side of \eqref{sexpistols}.

\subsubsection{First estimate}
Using the support properties of the paraproduct decomposition as in \cref{support}, equality \eqref{supportT}, we write
\begin{equation}
 \langle \D_k \opphi (\TT_{r_i} u^1_j), \D_k \p_i\rphi_j \rangle
 =
 \sum_{|k'-k|\leq4} \langle \D_k\opphi (\Low_{k'-1} r_i) (\D_{k'}  u^1_j), \D_k \p_i\rphi_j \rangle
 .
\end{equation}
Thanks to the product estimate \eqref{eq:vandamme_fou} from \cref{vandamme}, the embedding estimate \eqref{eq:jus1} from \cref{lem:embedding}, the decomposition \eqref{SLOW} of $\Low_{k'-1}$ and estimate \eqref{eq:bern-ball} from \cref{lem:Bern}, and the definition \eqref{ell1} of $\ell_1$, we obtain
\begin{equation}
 \begin{split}
  | \langle \D_k\opphibeta (\Low_{k'-1} r_i) & (\D_{k'} u^1_j) , \D_k \p_i\rphi_j \rangle |
  \\
  & \leq \| \Low_{k'-1} \rphi_i^+ \|_{L^\infty_x(L^2_y)} \| \D_{k'} \opphi ( u^1_j)^+ \|_{L^2_x(L^\infty_y)} \| \D_k \p_i\rphi_j\|_{L^2}
  \\
  & \leq \cb^{\f32} \ell_1 2^{-\frac{k'}2} 
   \| \D_k \nabla\rphi \|_{L^2} \sum_{k''\leq k'-2}2^{\frac{k''}2}\| \D_{k''} \rphi\|_{L^2}
 .
 \end{split}
\end{equation}
Thanks to the Peter-Paul inequality, we deduce that
\begin{equation}
 \begin{split}
  | \langle & \D_k\opphibeta (\Low_{k'-1} r_i) (\D_{k'} u^1_j) , \D_k \p_i\rphi_j \rangle |
  \\ &
  \leq  
   \frac{1}{(9\cdot8\cdot60)^2} \| \D_k \nabla\rphi \|_{L^2}^2
   + \frac{1}{4} (9\cdot8\cdot60)^2 \cb^3 2^{k''-k'} \left(\sum_{k''\leq k'-2} \ell_1 \| \D_{k''} \rphi\|_{L^2}\right)^2
   .
 \end{split}
\end{equation}
Summing these estimates and applying Minkowsky's inequality leads to
\begin{equation} \label{marteau1}
 \begin{split}
  \sum_{k\in\Z} & \left(\int_0^t | \langle \D_k \opphibeta (\TT_{r_i}  u^1_j), \D_k \p_i\rphi_j \rangle | \right)^{\f12}
  \\
   & \leq
   2160 \cb^2 
  \sum_{k''\in\Z} \left( \sum_{k' \geq k''-2} \sum_{|k-k'|\leq4} 2^{\frac{k''-k'}{2}} \right) 
  \left( \int_0^t \ell_1^2 \| \D_{k''} \rphi \|_{L^2}^2 \right)^{\f12}
  \\
  & \quad + 
   \frac{1}{9\cdot8\cdot60} \sum_{k\in\Z} \sum_{|k'-k|\leq4} \left(\int_0^t \| \D_k \na \rphi \|_{L^2}^2 \right)^{\f12}
  \\
  & \leq
   2\cdot10^6 \cb^2 
  \sum_{k\in\Z}
  \left( \int_0^t \ell_1^2 \| \D_{k} \rphi \|_{L^2}^2 \right)^{\f12}
  +
   \frac{1}{8\cdot60} \sum_{k\in\Z} \left(\int_0^t \| \D_k \na \rphi \|_{L^2}^2 \right)^{\f12}
   .
 \end{split}
\end{equation}

\subsubsection{Second estimate}
Using the support properties of the paraproduct decomposition as in \cref{support}, equality \eqref{supportT}, we write
\begin{equation*}
 \langle \D_k \opphi (\TT_{u^1_j} r_i), \D_k \p_i\rphi_j \rangle
 =
 \sum_{|k'-k|\leq4} \langle \D_k\opphi (\Low_{k'-1} u^1_j) (\D_{k'} r_i), \D_k \p_i\rphi_j \rangle
 .
\end{equation*}
It follows from the decomposition \eqref{SLOW} of $\Low_{k'-1}$, the estimate \eqref{embed2} from \cref{lem:embedding} and the definition \eqref{ell1} of $\ell_1$ that
\begin{equation}
 \|  \opphi \Low_{k'-1}(u^1)^+ \|_{L^\infty}\leq \cb^2 \ell_1
 .
\end{equation}
Hence, with the product estimate \eqref{eq:vandamme_fou} from \cref{vandamme}, we infer
\begin{equation}
 \begin{split}
  | \langle\D_k \opphibeta (\Low_{k'-1}  u^1_j) & (\D_{k'} r_i) , \D_k \p_i\rphi_j \rangle |
  \\
  & \leq \| \opphi \Low_{k'-1} (u^1)^+ \|_{L^\infty} \| \D_{k'} \rphi_i^+ \|_{L^2} \| \D_k \nabla\rphi\|_{L^2}
  \\
  & \leq \cb^2 \ell_1 \| \D_{k'} \rphi \|_{L^2}\| \D_{k} \nabla\rphi \|_{L^2}.
 \end{split}
\end{equation}
Summing these estimates and using the Peter-Paul inequality yields
\begin{equation} \label{marteau2}
 \begin{split}
  \sum_{k\in\Z} & \left(\int_0^t | \langle \D_k \opphibeta (\TT_{u^1_j} r_i), \D_k \p_i\rphi_j \rangle | \right)^{\f12}
  \\
  & \leq 2\cdot 10^5 \cb^2
  \sum_{k\in\Z}  \left(\int_0^t \ell_1^2 \| \D_{k} \rphi\|_{L^2}^2\right)^{\frac12}
  +\frac1{8\cdot60}\sum_{k\in\Z}\left(\int_0^t
   \| \D_k \nabla\rphi \|_{L^2}^2\right)^{\frac12}.
 \end{split}
\end{equation}
%

\subsubsection{Third estimate}
By using the support properties of the paraproduct decomposition as in \cref{support}, equality \eqref{supportR}, and the definition \eqref{eq:wt-delta}, we obtain
\begin{equation*}
 \begin{split}
 \langle \D_k \opphi \RR(r_i,  u^1_j), \D_k \p_i\rphi_j \rangle
  & =
 \sum_{k'\geq k-3} \langle \D_k\opphi (\D_{k'} r_i) (\wt{\D}_{k'}  u^1_j), \D_k \p_i\rphi_j \rangle
  \\
  & = \sum_{k'\geq k-3} \sum_{|k''-k'|\leq1}
  \langle \D_k\opphi (\D_{k'} r_i) (\D_{k''}  u^1_j), \D_k \p_i\rphi_j \rangle
  .
 \end{split}
\end{equation*}
Thanks to the product estimate \eqref{eq:vandamme_fou} from \cref{vandamme}, the embedding estimate \eqref{eq:jus1} from \cref{lem:embedding}, estimate \eqref{eq:bern-ball} from \cref{lem:Bern}, and the definition \eqref{ell1} of $\ell_1$, we infer
\begin{equation}
 \begin{split}
  | \langle \D_k\opphibeta (\D_{k'} r_i) & (\D_{k''}  u^1_j) , \D_k \p_i\rphi_j \rangle |
  \\
  & \leq \| \D_{k'} \rphi_i^+ \|_{L^\infty_x(L^2_y)} \| \opphi \D_{k''} (u^1_j)^+ \|_{L^2_x(L^\infty_y)} \| \D_k \nabla\rphi \|_{L^2}
  \\
  & \leq  2^{\frac{k-k''}2} \cb^{\f32} \ell_1 \| \D_{k'} \rphi \|_{L^2} \| \D_k \nabla\rphi \|_{L^2}.
 \end{split}
\end{equation}
Using the Peter-Paul inequality and summing the resulting estimates, we find
\begin{equation} \label{marteau3}
 \begin{split}
  \sum_{k\in\Z} & \left(\int_0^t | \langle \D_k \opphibeta \RR(r_i, u^1_j), \D_k \p_i\rphi_j \rangle | \right)^{\f12}
 \\
  \leq &
  10^6 \cb^2 
  \sum_{k\in\Z}  \left(\int_0^t \ell_1^2 \| \D_{k} \rphi \|_{L^2}^2 \right)^{\f12}+
 \f1{8\cdot60}\sum_{k\in\Z} \left(\int_0^t
   \| \D_k \nabla\rphi \|_{L^2}^2\right)^{\frac12}.
 \end{split}
\end{equation}

\medskip

Eventually, summing estimates \eqref{marteau1}, \eqref{marteau2} and \eqref{marteau3} for the 8 pairs of indexes concludes the proof of \cref{lem:estim-ampli} with the claimed constants.


\subsection{Estimate of \texorpdfstring{$ (r\cdot\na) r $}{the nonlinearity}. Proof of Lemma \ref{lem:trilinear}}
\label{trili}

We prove \cref{lem:trilinear}, concerning the estimate of the trilinear term. First, using Bony's paraproduct decomposition \eqref{Bony} and the divergence free condition in \eqref{eq.nse.r}, we write the quadratic term as
\begin{equation} \label{sextuples}
 (r\cdot\na) r= \TT_{\jrun}\p_x r+\TT_{\p_x r}\jrun+\p_x \RR(\jrun,r)+\TT_{\jrdeux}\p_yr+\TT_{\p_yr}\jrdeux+\p_y \RR(\jrdeux,r). 
\end{equation}
As a shorthand, we set
\begin{equation}
 E := \sum_{k'\in\Z} 2^{\f{k'}2}  \left(\int_0^t e^{\beta_\star} \|\na \rphi\|_{\besov{0}} \|\D_{k'} \rphi \|^2_{L^2} \right)^\frac12,
\end{equation}    
which appears in the right hand side of the estimate given in \cref{lem:trilinear}. The proof of \cref{lem:trilinear} is obtained by summation of the six following estimates:
\begin{align}
\label{tri1}
\sum_{k\in\Z}  \left(\int_0^t\bigl|\langle \D_k  \opphibeta(\TT_{\jrun}\p_xr), \D_k\rphi\rangle\bigr|\right)^{\f12}
 &\leq 10^2 \cb^2 E  ,
  \\   \label{tri2}
 \sum_{k\in\Z}  \left(\int_0^t\bigl|\langle\D_k  \opphibeta (\TT_{\p_xr}\jrun),\D_k\rphi\rangle\bigr|\right)^{\f12}
 &\leq 10^2 \cb^2 E ,
  \\ \label{tri3}
\sum_{k\in\Z}  \left(  \int_0^t\bigl|\langle\D_k \opphibeta \p_x \RR(\jrun,r),\D_k\rphi\rangle\bigr| \right)^{\f12} 
&\leq  10^2 \cb^2 E ,
 \\  \label{tri4}
\sum_{k\in\Z}  \left(\int_0^t\bigl|\langle\D_k \opphibeta(\TT_{\jrdeux}\p_yr),\D_k\rphi\rangle\bigr| \right)^\frac12
 &\leq 10^2 \cb^2 E ,
 \\  \label{tri5} \sum_{k\in\Z} \left(\int_0^t\bigl|\langle\D_k\opphibeta(\TT_{\p_yr}\jrdeux),\D_k\rphi\rangle\bigr| \right)^\frac12 
 &\leq 10^2 \cb^2 E ,
 \\ \label{tri6}
\sum_{k\in\Z}  \left( \int_0^t\bigl|\langle\D_k \opphibeta\p_y \RR(\jrdeux,r),\D_k\rphi\rangle\bigr| \right)^{\f{1}2}
& \leq 
10^2 \cb^2 E .
 \end{align}


\subsubsection{Proof of \eqref{tri1}}

Using the support properties of the paraproduct decomposition as in \cref{support}, equality \eqref{supportT}, we expand
\begin{equation}
  \langle \D_k \opphi (\TT_{r_1} \p_x r), \D_k \rphi \rangle
  = 
  \sum_{|k'-k|\leq 4} \langle \opphi (\Low_{k'-1} r_1) (\D_{k'} \p_x r) , \D_k \rphi \rangle
  .
\end{equation}
Thanks to the product estimate \eqref{eq:vandamme_fou} from \cref{vandamme}, we obtain
\begin{equation}
  | \langle \opphibeta (\Low_{k'-1} r_1) (\D_{k'} \p_x r) , \D_k \rphi \rangle |
  \leq e^{\beta_\star}
  \| \Low_{k'-1} \run^+ \|_{L^\infty}
  \| \D_{k'} \p_x \rphi^+ \|_{L^2}
  \| \D_k \rphi^+ \|_{L^2}
  .
\end{equation}
Thanks to estimate \eqref{eq:bern-ball} from \cref{lem:Bern},
\begin{equation}
 \| \D_{k'} \p_x \rphi \|_{L^2}
 \leq 
 \cb 2^{k'} \| \D_{k'} \rphi \|_{L^2}
 .
\end{equation}
Using estimate \eqref{embed2} from \cref{lem:embedding}, we obtain that
\begin{equation}
 \| \Low_{k'-1} \run^+\|_{L^\infty}
 \leq 
 \cb^2 \|\na \run^+\|_{\besov{0}} 
 \leq 
 \cb^2 \|\na \rphi\|_{\besov{0}}
 .
\end{equation}
Gathering our estimates, we obtain
\begin{equation}
 \begin{split}
  & \sum_{k\in\Z} \left( \int_0^t | \langle \D_k \opphibeta (\TT_{r_1} \p_x r), \D_k \rphi \rangle | \right)^{\f12}
  \\
  & \leq \cb^{\f32} \sum_{k\in\Z} \sum_{|k'-k|\leq4}
  2^{\f{k'}{2}} \left( \int_0^t e^{\beta_\star} \| \nabla \rphi \|_{\besov{0}} \|\D_{k'} \rphi\|_{L^2} \|\D_{k} \rphi\|_{L^2} \right)^{\f12}
  \\
  & \leq 10^2 \cb^2 \sum_{k'\in\Z} 2^{\f{k'}{2}} 
  \left( \int_0^t e^{\beta_\star} \| \nabla \rphi \|_{\besov{0}} \|\D_{k'} \rphi\|_{L^2}^2 \right)^{\f12}.
 \end{split}
\end{equation}
This concludes the proof of \eqref{tri1}.


\subsubsection{Proof of \eqref{tri2}}
Using the support properties of the paraproduct decomposition as in \cref{support}, equality \eqref{supportT}, we expand
\begin{equation}
  \langle \D_k \opphi (\TT_{\p_x r} r_1), \D_k \rphi \rangle
  = 
  \sum_{|k'-k|\leq 4} \langle \opphi (\Low_{k'-1} \p_x r) (\D_{k'} r_1) , \D_k \rphi \rangle
  .
\end{equation}
Thanks to the product estimate \eqref{eq:vandamme_fou} from \cref{vandamme}, we obtain
\begin{equation}
 \begin{split}
  | \langle \opphibeta (\Low_{k'-1} \p_x r) & (\D_{k'} r_1) , \D_k \rphi \rangle |
  \\
  & \leq e^{\beta_\star}
  \| \Low_{k'-1} \p_x \rphi^+ \|_{L^\infty_x(L^2_y)}
  \| \D_{k'} \run^+ \|_{L^2_x(L^\infty_y)}
  \| \D_k \rphi^+ \|_{L^2}
  .
 \end{split}
\end{equation}
Thanks to estimate \eqref{embed1} from \cref{lem:embedding}, we get
\begin{equation} \label{crude1}
 \|\D_{k'} \run^+\|_{L^2_x(L^\infty_y)}
 \leq 
 \cb 2^{-\f{{k'}}2} \|\na \rphi\|_{\besov{0}}
 .
\end{equation}
Gathering our estimates and using the Peter-Paul inequality, 
\begin{equation}
 \begin{split}
&   \left( \int_0^t   \| \Low_{k'-1} \p_x \rphi^+ \|_{L^\infty_x(L^2_y)}
  \|\na \rphi\|_{\besov{0}}
  \| \D_k \rphi^+ \|_{L^2}  \right)^{\f12}
   \leq 
  2^{\f{3k}{4}} \left( \int_0^t \| \nabla \rphi \|_{\besov{0}} \|\D_{k} \rphi\|_{L^2}^2 \right)^{\f12}
\\ &\quad\quad  \quad  \quad \quad\quad\quad  \quad  \quad \quad  + 2^{-\f{3k}{4}}  \left( \int_0^t \| \nabla \rphi \|_{\besov{0}} \| \Low_{k'-1} \p_x \rphi^+ \|_{L^\infty_x(L^2_y)}^2 \right)^{\f12} ,
 \end{split}
\end{equation}
 we obtain
\begin{equation}
\label{dim}
 \begin{split}
  & \sum_{k\in\Z} \left( \int_0^t | \langle \D_k \opphibeta (\TT_{\p_x r} r_1), \D_k \rphi \rangle | \right)^{\f12}
  \\
  & \leq \cb^{\f12} \sum_{k\in\Z} 2^{\f{k}2} \sum_{|k'-k|\leq4}
  2^{\f{k-k'}{4}} \left( \int_0^t e^{\beta_\star} \| \nabla \rphi \|_{\besov{0}} \|\D_{k} \rphi\|_{L^2}^2 \right)^{\f12}
  \\
  & \quad + 
  \cb^{\f12} \sum_{k\in\Z} 2^{-k} \sum_{|k'-k|\leq4}
  2^{\f{k-k'}{4}} \left( \int_0^t e^{\beta_\star} \| \nabla \rphi \|_{\besov{0}} \| \Low_{k'-1} \p_x \rphi^+ \|_{L^\infty_x(L^2_y)}^2 \right)^{\f12}
  .
 \end{split}
\end{equation}
The first term is bounded by $10 \cb E$. Using Minkowski's inequality and \eqref{SLOW}, we estimate the second term as
\begin{equation}
 \begin{split}
  & \sum_{k\in\Z} 2^{-k} \sum_{|k'-k|\leq4}
  2^{\f{k-k'}{4}} \left( \int_0^t e^{\beta_\star} \| \nabla \rphi \|_{\besov{0}} \| \Low_{k'-1} \p_x \rphi^+ \|_{L^\infty_x(L^2_y)}^2 \right)^{\f12}
  \\
  & \leq
  \sum_{k\in\Z} 2^{-k} \sum_{|k'-k|\leq4} 2^{\f{k-k'}{4}}  \sum_{j\leq k'-2}
  \left( \int_0^t e^{\beta_\star} \| \nabla \rphi \|_{\besov{0}} \| \D_{j} \p_x \rphi^+ \|_{L^\infty_x(L^2_y)}^2 \right)^{\f12}
  .
 \end{split}
\end{equation}
Now we deduce from Bernstein's \cref{lem:Bern}, estimate \eqref{eq:bern-ball}, that
\begin{equation}
 \|\D_j \p_x \rphi^+\|_{L^\infty_x(L^2_y)}^2
 \leq 
 \cb^2 2^{3j}\|\D_j \rphi\|_{L^2}^2
 .
\end{equation}
We deduce from these estimates that the second term of \eqref{dim} is bounded by $80 \cb^2 E$. This concludes the proof of \eqref{tri2}.


\subsubsection{Proof of \eqref{tri3}}
First, using integration by parts in the horizontal direction, we get
\begin{equation}
 \langle \D_k \opphi \p_x \RR(\jrun,r) , \D_k \rphi \rangle
 = 
 - \langle \D_k \opphi \RR(\jrun,r) , \D_k \p_x \rphi \rangle
 .
\end{equation}
Using the support properties of the paraproduct decomposition as in \cref{support}, equality \eqref{supportR}, and the definition \eqref{eq:wt-delta} we expand the term as
\begin{equation}
 \begin{split}
  \langle \D_k \opphi \RR(\jrun,r) , \p_x \D_k \rphi \rangle
  & = \sum_{k'\geq k-3} \langle \opphi (\D_{k'} \jrun) (\wt{\D}_{k'} r) , \D_k \p_x \rphi \rangle
  \\
  & = \sum_{k'\geq k-3} \sum_{|k''-k'|\leq1} \langle \opphi (\D_{k'} \jrun) (\D_{k''} r) , \D_k \p_x \rphi \rangle
  .
 \end{split}
\end{equation}
Thanks to the product estimate \eqref{eq:vandamme_fou} from \cref{vandamme}, we obtain
\begin{equation}
 \begin{split}
  | \langle \opphibeta (\D_{k'} \jrun) & (\D_{k''} r) , \D_k \partial_x \rphi \rangle |
  \\
  & \leq e^{\beta_\star}
 \|\D_{k'} \run^+\|_{L^2_x(L^\infty_y)} 
 \|\D_{k''} \rphi^+\|_{L^2_x(L^2_y)}
 \|\D_k\p_x\rphi^+\|_{L^\infty_x(L^2_y)}
 .
 \end{split}
\end{equation}
Thanks to estimate \eqref{embed1} from \cref{lem:embedding}, we have
\begin{equation}
 \|\D_{k'} \run^+\|_{L^2_x(L^\infty_y)} 
 \leq
 \cb 2^{-\f{k'}2} \| \nabla \rphi \|_{\besov{0}}
 .
\end{equation}
Thanks to estimate \eqref{eq:bern-ball} from \cref{lem:Bern}, we have
\begin{equation}
 \|\D_k\p_x\rphi^+\|_{L^\infty_x(L^2_y)}
 \leq
 \cb 2^{\f{3k}{2}} \|\D_k \rphi\|_{L^2}
 .
\end{equation}
Gathering our estimates, we obtain
\begin{equation}
 \begin{split}
  & \sum_{k\in\Z} \left( \int_0^t | \langle \D_k \opphibeta \p_x \RR(\jrun,r) , \D_k \rphi \rangle | \right)^{\f12}
  \\
  & \leq \cb \sum_{k\in\Z} \sum_{k'\geq k-3} \sum_{|k''-k|\leq1}
  2^{\f{3k}{4}-\f{k'}{4}} \left( \int_0^t e^{\beta_\star} \| \nabla \rphi \|_{\besov{0}} \|\D_{k} \rphi\|_{L^2} \|\D_{k''} \rphi\|_{L^2} \right)^{\f12}
  \\
  & \leq \cb \sum_{k\in\Z} \sum_{k'\geq k-3} \sum_{|k''-k|\leq1}
  2^{\f{3k}{4}-\f{k'}{4}} \left( \int_0^t e^{\beta_\star} \| \nabla \rphi \|_{\besov{0}} \|\D_{k} \rphi\|_{L^2}^2 \right)^{\f12}
  \\
  & \quad + \cb \sum_{k\in\Z} \sum_{k'\geq k-3} \sum_{|k''-k|\leq1}
  2^{\f{3k}{4}-\f{k'}{4}} \left( \int_0^t e^{\beta_\star} \| \nabla \rphi \|_{\besov{0}} \|\D_{k''} \rphi\|_{L^2}^2 \right)^{\f12}
 \end{split}
\end{equation}
Up to reordering the sums, one gets that the first term is bounded by $32 \cb E$ and that the second term is bounded by $51 \cb E$. This concludes the proof of \eqref{tri3}.


\subsubsection{Proof of \eqref{tri4}}
Using the support properties of the paraproduct decomposition as in \cref{support}, equality \eqref{supportT}, we expand
\begin{equation}
  \langle \D_k \opphi (\TT_{\jrdeux} \partial_y r), \D_k \rphi \rangle
  = 
  \sum_{|k'-k|\leq 4} \langle \opphi (\Low_{k'-1} \jrdeux) (\D_{k'} \partial_y r) , \D_k \rphi \rangle
  .
\end{equation}
Thanks to the product estimate \eqref{eq:vandamme_fou} from \cref{vandamme}, we obtain
\begin{equation}
  | \langle \opphibeta (\Low_{k'-1} \jrdeux) (\D_{k'} \partial_y r) , \D_k \rphi \rangle |
  \leq e^{\beta_\star}
  \| \Low_{k'-1} \rdeux^+ \|_{L^\infty}
  \| \D_{k'} \partial_y \rphi^+ \|_{L^2}
  \| \D_k \rphi^+ \|_{L^2}
  .
\end{equation}
From the definition \eqref{eq:besov} of $\besov{0}$,
\begin{equation}
 \| \D_{k'} \partial_y \rphi \|_{L^2}
 \leq 
 \| \nabla \rphi \|_{\besov{0}}.
\end{equation}
Summing over $k \in \Z$ and using Young's inequality, we obtain
\begin{equation} \label{eq:tri4_1}
 \begin{split}
  & \sum_{k\in\Z} \left( \int_0^t | \langle \D_k \opphibeta (\TT_{r_2}\p_y r), \D_k \rphi \rangle | \right)^{\f12}
  \\
  & \leq 
  \sum_{k\in\Z} \sum_{|k'-k|\leq4} 2^{-\f{k}2} \left( \int_0^t e^{\beta_\star}
  \| \nabla \rphi \|_{\besov{0}}
  \| \Low_{k'-1} \rdeux^+ \|_{L^\infty}^2
  \right)^{\f12} 
  \\
  & \quad +
  \sum_{k\in\Z} \sum_{|k'-k|\leq4} 2^{\f{k}2} \left( \int_0^t e^{\beta_\star}
  \| \nabla \rphi \|_{\besov{0}}
  \| \D_k \rphi \|_{L^2}^2 \right)^{\f12}
 \end{split}
\end{equation}
Using Minkowski's inequality and \eqref{SLOW}, the first sum is bounded as
\begin{equation} \label{eq:tri4_2}
 \begin{split}
  \sum_{k\in\Z} & \sum_{|k'-k|\leq4} 2^{-\f{k}2} \left( \int_0^t e^{\beta_\star}
  \| \nabla \rphi \|_{\besov{0}}
  \| \Low_{k'-1} \rdeux^+ \|_{L^\infty}^2
  \right)^{\f12}
  \\
  & \leq 
  \sum_{k\in\Z} \sum_{|k'-k|\leq4} 2^{-\f{k}2}  \sum_{j\leq k'-2} \left( \int_0^t e^{\beta_\star}
  \| \nabla \rphi \|_{\besov{0}}
  \| \D_j \rdeux^+ \|_{L^\infty}^2
  \right)^{\f12}
 \end{split}
\end{equation}
Thanks to estimates \eqref{eq:tg} from \cref{lem:tg} and \eqref{eq:bern-ball} from \cref{lem:Bern}, we get
\begin{equation}
 \|\D_j \rdeux\|_{L^\infty} 
 \leq 
 \cb^2 2^{j}  \|\D_j \rphi\|_{L^2} 
 .
\end{equation}
Gathering these two estimates and reordering the sums we conclude that the first term in the right hand side of 
\eqref{eq:tri4_2} is bounded by $36 \cb^2 E$. 
Since the second term  is bounded by $9E$ this concludes 
the proof of \eqref{tri4}.


\subsubsection{Proof of \eqref{tri5}}
Using the support properties of the paraproduct decomposition as in \cref{support}, equality \eqref{supportT}, we expand
\begin{equation}
  \langle \D_k \opphi (\TT_{\p_y r}r_2), \D_k \rphi \rangle
  = 
  \sum_{|k'-k|\leq 4} \langle \opphi (\Low_{k'-1} \p_y r) (\D_{k'} \jrdeux) , \D_k \rphi \rangle
  .
\end{equation}
Thanks to the product estimate \eqref{eq:vandamme_fou} from \cref{vandamme}, we obtain
\begin{equation}
  | \langle \opphibeta (\Low_{k'-1} \p_y r) (\D_{k'} \jrdeux) , \D_k \rphi \rangle |
  \leq e^{\beta_\star}
  \| \Low_{k'-1} \partial_y \rphi^+ \|_{L^2}
  \| \D_{k'} \rdeux^+ \|_{L^\infty}
  \| \D_k \rphi^+ \|_{L^2}
  .
\end{equation}
On the one hand, thanks to estimate \eqref{eq:tg} from \cref{lem:tg} and estimate \eqref{eq:bern-ball} from \cref{lem:Bern}, we have
\begin{equation} 
 \| \D_{k'} \rdeux^+ \|_{L^\infty}
 \leq
 \cb^2 2^{k'} \| \D_{k'} \rphi \|_{L^2}
 .
\end{equation}
On the other hand, from the definition \eqref{eq:besov} of $\besov{0}$,
\begin{equation}
 \| \Low_{k'-1} \partial_y \rphi \|_{L^2}
 \leq 
 \| \nabla \rphi \|_{\besov{0}}.
\end{equation}
Gathering our estimates, we obtain
\begin{equation}
 \begin{split}
  & \sum_{k\in\Z} \left( \int_0^t | \langle \D_k \opphibeta (\TT_{\p_y r}r_2), \D_k \rphi \rangle | \right)^{\f12}
  \\
  & \leq \cb \sum_{k\in\Z} \sum_{|k'-k|\leq4}
  2^{\f{k'}{2}} \left( \int_0^t e^{\beta_\star} \| \nabla \rphi \|_{\besov{0}} \|\D_{k'} \rphi\|_{L^2} \|\D_{k} \rphi\|_{L^2} \right)^{\f12}
  \\
  & \leq 10^2 \cb \sum_{k'\in\Z} 2^{\f{k'}{2}} 
  \left( \int_0^t e^{\beta_\star} \| \nabla \rphi \|_{\besov{0}} \|\D_{k'} \rphi\|_{L^2}^2 \right)^{\f12}.
 \end{split}
\end{equation}
This concludes the proof of \eqref{tri5}.


\subsubsection{Proof of \eqref{tri6}}
First, using integration by parts in the vertical direction, and the null boundary condition in \eqref{eq.nse.rPhi}, we get
\begin{equation}
 \langle \D_k \opphi \p_y \RR(\jrdeux,r), \D_k\rphi \rangle 
 = 
 - \langle \D_k \opphi \RR(\jrdeux,r), \D_k\p_y \rphi \rangle
\end{equation}
Using the support properties of the paraproduct decomposition as in \cref{support}, equality \eqref{supportR}, and the definition \eqref{eq:wt-delta} we expand the term as
\begin{equation}
 \begin{split}
  \langle\D_k \opphi \RR(\jrdeux,r),\D_k\p_y \rphi\rangle
  & = \sum_{k'\geq k-3} \langle \opphi (\D_{k'} \jrdeux) (\wt{\D}_{k'} r) ,\D_k \partial_y \rphi \rangle
  \\
  & = \sum_{k'\geq k-3} \sum_{|k''-k'|\leq1} \langle \opphi (\D_{k'} \jrdeux) (\D_{k''} r) ,\D_k \partial_y \rphi \rangle
 \end{split}
\end{equation}
Thanks to the product estimate \eqref{eq:vandamme_fou} from \cref{vandamme}, we obtain
\begin{equation}
 \begin{split}
  | \langle \opphibeta (\D_{k'} \jrdeux) & (\D_{k''} r) , \D_k \partial_y \rphi \rangle |
  \\
  & \leq e^{\beta_\star}
 \|\D_{k'} \rdeux^+\|_{L^2_x(L^\infty_y)} 
 \|\D_{k''} \rphi^+\|_{L^2_x(L^2_y)}
 \|\D_k\p_y\rphi^+\|_{L^\infty_x(L^2_y)}.
 \end{split}
\end{equation}
Thanks to estimates \eqref{eq:tg} from \cref{lem:tg} and \eqref{eq:bern-ball} from \cref{lem:Bern}, we get
\begin{equation}
 | \langle \opphibeta (\D_{k'} \jrdeux) (\D_{k''} r) ,\D_k \partial_y \rphi \rangle |
 \leq e^{\beta_\star}
 \cb^2 2^{\f{k}2+\f{k'}2} 
 \|\D_{k'} \rphi\|_{L^2} \|\D_{k''} \rphi\|_{L^2}
 \|\D_k\p_y\rphi\|_{L^2}.
\end{equation}
We use the following crude estimate, which follows from definition \eqref{eq:besov}.
\begin{equation}
 \|\D_k\p_y\rphi\|_{L^2} \leq \| \nabla \rphi \|_{\besov{0}}.
\end{equation}
Gathering our estimates, we obtain
\begin{equation}
 \begin{split}
  & \sum_{k\in\Z} \left( \int_0^t | \langle \D_k \opphibeta \p_y \RR(\jrdeux,r), \D_k\rphi \rangle | \right)^{\f12}
  \\
  & \leq \cb \sum_{k\in\Z} \sum_{k'\geq k-3} \sum_{|k''-k'|\leq1}
  2^{\f{k}{4}+\f{k'}{4}} \left( \int_0^t e^{\beta_\star}\| \nabla \rphi \|_{\besov{0}} \|\D_{k'} \rphi\|_{L^2} \|\D_{k''} \rphi\|_{L^2} \right)^{\f12}
  \\
  & \leq 6 \cb \sum_{k\in\Z} \sum_{k'\geq k-3}
  2^{\f{k}{4}+\f{k'}{4}} \left( \int_0^t e^{\beta_\star}\| \nabla \rphi \|_{\besov{0}} \|\D_{k'} \rphi\|_{L^2}^2 \right)^{\f12}
  \\
  & \leq 6 \cb \sum_{k'\in\Z} 2^{\f{k'}{2}} \left(\sum_{k\leq k'+3} 2^{\f{k-k'}{4}}\right)
  \left( \int_0^t e^{\beta_\star}\| \nabla \rphi \|_{\besov{0}} \|\D_{k'} \rphi\|_{L^2}^2 \right)^{\f12}
  \\
  & \leq 10^2 \cb \sum_{k'\in\Z} 2^{\f{k'}{2}} 
  \left( \int_0^t e^{\beta_\star} \| \nabla \rphi \|_{\besov{0}} \|\D_{k'} \rphi\|_{L^2}^2 \right)^{\f12}.
 \end{split}
\end{equation}
This concludes the proof of \eqref{tri6}.

\section{Analytic estimates for the approximate trajectories}
\label{sec:sources}

\subsection{Preliminary estimates}

We introduce notations and prove preliminary estimates that will be used in the sequel. In this paragraph, $a$ denotes a function in $L^2(\band)$ for which all the norms and sums that we manipulate are finite.

\newcommand{\grhob}{\mathcal{G}^{\rho_b}}

\begin{itemize}

\item For $s \in \N$, $\rho > 0$ and $p \in \{2,+\infty\}$, we introduce the tangential analytic-type norm 
\begin{equation} \label{Gsp}
 \mathcal{G}^\rho_{s,p}(a) := \sum_{0 \leq \alpha + \beta \leq s} \enskip  \sup_{m \geq 0} \enskip
 \frac{\rho^m}{m!} \| \partial_x^{m} \p_x^{\alpha} \partial_y^{\beta} a \|_{L^2_x(L^p_y)}.
\end{equation}
In particular, thanks to condition \eqref{ub.anal}, one has $\grhob_{3,2}(u_b) \leq C_b$.

\item We introduce the notations $u^1_i = \nabla^\perp [\chi_\delta \psi_b]$ and $u^1_o := \nabla^\perp [(1-\chi_\delta) \psi_b]$. Hence, recalling the translation notation \eqref{trans}, the definition \eqref{eq:magic-u1} of $u^1$ can be recast as
\begin{equation} \label{u1io}
 u^1 = \beta \tau_h u^1_i + \tau_h u^1_o.
\end{equation}
From the definition \eqref{psi_b} of $\psi_b$, the smoothness of $\chi_\delta$ and the conditions \eqref{ub.anal} we deduce that there exists $C_1$ such that,
\begin{equation} \label{u1io.anal}
 \grhob_{3,2}(u^1_i) + 
 \grhob_{3,2}(u^1_o)
 \leq C_1.
\end{equation}
Thus, using the relation \eqref{u1io}, for any $t \geq 0$, 
\begin{equation} \label{u1.anal}
 \grhob_{3,2}(u^1(t)) \leq C_1.
\end{equation}

\item Let $p \in \{2,+\infty\}$. Thanks to the Peter-Paul inequality, then using the property \eqref{philp3} and the support property \eqref{supp.philp} for the operator~$\D_k$,
 \begin{equation}
  \begin{split} 
   \sum_{k\in\Z} \| \oprhozero \D_k a \|_{L^2_x(L^p_y)}
   & \leq \Big(\sum_{k\in\Z} 2^{-\frac{|k|}2}\Big)^{\frac12}
   \Big(\sum_{k\in\Z} 2^{\frac{|k|}2} \| \oprhozero \D_k a \|_{L^2_x(L^p_y)}^2 \Big)^{\frac12}
   \\ & \leq
   \frac{3}{2\pi^2} \Big( \int (|\xi|^{\frac12}+|\xi|^{-\frac12}) e^{2\rho_0|\xi|} \| \cF a (\xi) \|_{L^p_y}^2 \dd \xi  \Big)^{\frac12}
  \end{split}
 \end{equation}
 For low frequencies, using a uniform bound for the Fourier transform yields
 \begin{equation}
  \int_{|\xi|\leq1} (|\xi|^{\frac12}+|\xi|^{-\frac12}) e^{\rho_0|\xi|} \| \cF a (\xi) \|_{L^p_y}^2 \dd \xi
  \leq \frac{16}{3} e^{\rho_0} \| a \|_{L^1_x(L^p_y)}^2.
 \end{equation}
 For high frequencies, using the elementary inequality $e^x \leq 2 \cosh x$,
 \begin{equation}
  \int_{|\xi|\geq1} (|\xi|^{\frac12}+|\xi|^{-\frac12}) e^{\rho_0|\xi|} \| \cF a (\xi) \|_{L^p_y}^2 \dd \xi
  \leq 16 \pi^2 \sum_{m\geq0} \frac{(2\rho_0)^{2m}}{(2m)!} \|\partial_x^m \partial_x a \|_{L^2_x(L^p_y)}^2.
 \end{equation}
 Hence, for any $\rho > \rho_0$, there exists ${C}_\rho >0$ such that,
 \begin{equation} \label{plate}
  \sum_{k\in\Z} \| \oprhozero \D_k a \|_{L^2_x(L^p_y)}
  \leq
  {C}_\rho \left(\|a\|_{L^1_x(L^p_y)} + \mathcal{G}^\rho_{1,p}(a) \right).
 \end{equation}

\end{itemize}

\subsection{Estimates for the amplification terms}
\label{sec:beta}

We prove \cref{thm:beta}. Recalling the definition \eqref{lstar} of $\beta$, we proceed term by term.
\begin{itemize}

\item Recalling the definition \eqref{def.Hsm} of the space $H^{1,1}(\R_+)$ and using the decay estimate~\eqref{ineq.V.decay} for the boundary layer from \cref{Lemma:V} with $n = 1$, we get 
\begin{equation}
 \begin{split}
  \int_0^{+\infty} \| \chi' \eval{v^0(s)} \|_{L^\infty(\band)} \dd s
  & \leq \|\chi'\|_{L^\infty(-1,1)} \int_0^{+\infty} \| V(s) \|_{L^\infty(\R_+)} \dd s
  \\
  & \leq 2 \|\chi'\|_{L^\infty(-1,1)} \int_0^{+\infty} \| V(s) \|_{H^{1,1}(\R_+)} \dd s
  \\
  & \leq 2 C \|\chi'\|_{L^\infty(-1,1)} \int_0^{+\infty} \left| \frac{\ln(2+s)}{2+s} \right|^{\frac{5}{4}} \dd s
  \\
  & \leq 52 C \|\chi'\|_{L^\infty(-1,1)}.
 \end{split}
\end{equation}

\item Recalling the definition \eqref{eq.wW} of $w^\e$ and using estimate \eqref{estimate.W} from \cref{lemma.W},
\begin{equation}
 \e^2 \int_0^{T/\e^\kappa} \| \partial_y w^\e(t) \|_{L^\infty(\band)} \dd t
 \leq T \e^{2-\kappa} \| \p_y W^\e \|_{L^\infty(\R_+;L^\infty(\band))}
 \leq C_W T \e^{1-\kappa+\frac14}.
\end{equation}

\item Recalling the definition \eqref{ell1} of $\ell_1$ and applying estimate \eqref{plate} to $a=\nabla u^1(t)$ and \eqref{u1.anal}, there holds for any $t\geq 0$,
\begin{equation}
 \begin{split}
 \ell_1(t) 
 & \leq {C}_{\rho_b} \left( \| \nabla u^1(t) \|_{L^1_x(L^2_y)} + \grhob_{1,2}(\nabla u^1(t)) \right) 
  \\ & \leq {C}_{\rho_b} \left( \| \nabla u^1_i \|_{L^1_x(L^2_y)} + \| \nabla u^1_o \|_{L^1_x(L^2_y)} + C_1 \right)
  \\ & \leq {C}_{\rho_b} \left( 2 C_b + C_1 \right),
 \end{split}
\end{equation}
thanks to estimate \eqref{ub.int}.

\item Recalling the definition \eqref{defalpha} of $\alpha_\e$, 
\begin{equation}
 \int_0^{T/\e^\kappa} \alpha_\e(t) \dd t 
 \leq T + \frac{\pi}{2}.
\end{equation}

\end{itemize}
Gathering these four estimates concludes the proof of \cref{thm:beta}.

\subsection{Estimates for the source terms}
\label{sec:cf}

We prove \cref{thm:cf}. We will use the two following inequalities. First, for $g \in L^2((0,T/\e^\kappa)\times\band)$, thanks to the first part of the definition \eqref{defalpha} of $\alpha_\e$
\begin{equation} \label{type1}
  \sum_{k\in\Z} \left(\int_0^{T/\e^\kappa} \frac{1}{\alpha_\e} \| \oprhozero \D_k (\e g) \|_{L^2(\band)}^2 \right)^{\frac12}
  \leq \e^{1-\kappa} \sum_{k\in\Z} \sup_{t\geq0} \| \oprhozero \D_k g(t) \|_{L^2(\band)}.
\end{equation}
For $g \in L^2((0,T/\e^\kappa)\times\band)$ and $\mathcal{V}\in L^2(\R_+\times\R_+)$, thanks to the second part of the definition \eqref{defalpha} of $\alpha$ and estimate \eqref{eq.lemma.scale} from \cref{lemma.scale},
\begin{equation} \label{type2}
 \begin{split}
  \sum_{k\in\Z} & \left(\int_0^{T/\e^\kappa} \frac{1}{\alpha_\e} \| \oprhozero \D_k (\eval{\mathcal{V}} g) \|_{L^2(\band)}^2 \right)^{\frac12}
  \\
  & \leq 2 \e^{\frac14} \left(\int_{\R_+} (1+t^2) \| V(t) \|_{L^2(\R_+)}^2 \dd t\right)^{\frac12} \sum_{k\in\Z} \sup_{t\geq0} \| \oprhozero \D_k g(t) \|_{L^2_x(L^\infty_y)}.
 \end{split}
\end{equation}
Recalling the definition \eqref{eq:fapp} of $f^\e_{\mathrm{app}}$, we proceed slightly differently for the first four terms (using \eqref{type1}), and for the last three terms (using \eqref{type2}). 

\begin{equation} 
 \begin{split}
  f^\varepsilon_\mathrm{app}
 := &
 - \varepsilon \Delta u^1 + \varepsilon (u^1 \cdot \nabla) u^1
 {+\e^2 W^\e \partial_x u^1 + \e^2 (u^1 \cdot \ey) \partial_y w^\e}
 \\ & - \chi \eval{V} \partial_x u^1
 - \frac{u^1 \cdot \ey}{\varphi} \left(\sqrt{\varepsilon} \chi' \eval{z V} + \chi \varphi' \eval{z \partial_z V}\right) \ex.
 \end{split}
\end{equation}

\begin{itemize}

 \item \emph{First term}. We apply \eqref{type1} to $g := -\Delta u^1$. For $t \geq 0$, thanks to \eqref{u1io},
 \begin{equation}
  \begin{split}
   \| \oprhozero \D_k \Delta u^1(t) \|_{L^2(\band)}
   & = 
   \| \oprhozero \D_k \Delta (\beta(t) (\tau_h u^1_i)(t) + (\tau_h u^1_o)(t)) \|_{L^2(\band)} 
   \\ & =
   \| \oprhozero \D_k \Delta (\beta(t) u^1_i + u^1_o) \|_{L^2(\band)} 
   \\ & \leq 
   \| \oprhozero \D_k \Delta u^1_i \|_{L^2(\band)} +
   \| \oprhozero \D_k \Delta u^1_o \|_{L^2(\band)} 
  \end{split}
 \end{equation}
 Thanks to \eqref{plate}, \eqref{u1io.anal} and \eqref{ub.int}, we deduce that
 \begin{equation}
  \sum_{k\in\N} \| \oprhozero \D_k \Delta u^1_i \|_{L^2(\band)} +
   \| \oprhozero \D_k \Delta u^1_o \|_{L^2(\band)} 
   < +\infty.
 \end{equation}
 Hence,
 \begin{equation}
  \sum_{k\in\Z} \sup_{t\geq0} \| \oprhozero \D_k \Delta_u^1(t) \|_{L^2(\band)} < +\infty.
 \end{equation}

 \item \emph{Second term}. We apply \eqref{type1} to $g := (u^1\cdot\nabla) u^1$. Similarly, we decompose $u^1$ thanks to \eqref{u1io} to get an estimate which is uniform with respect to time. We get four terms. As an example, let us bound using \eqref{plate}, for some $\rho \in (\rho_0,\rho_b)$,
 \begin{equation}
  \sum_{k\in\Z} \| \oprhozero \D_k (u^1_i\cdot\nabla) u^1_i \|_{L^2(\band)} 
  \leq 
  \tilde{C}_\rho \left( \| (u^1_i\cdot\nabla) u^1_i \|_{L^1_x(L^2_y)} + \mathcal{G}^\rho_{1,2}((u^1_i\cdot\nabla) u^1_i) \right).
 \end{equation}
 The first term is finite thanks to \eqref{ub.int}. The second term requires an analytic estimate for this quadratic term. For $m \geq 0$, and $a,b \in L^2(\band)$, thanks the Leibniz differentiation rule,
 \begin{equation}
  \begin{split}
   \frac{\rho^m}{m!} \| \partial_x^m [(a\cdot\nabla)b]\|_{H^1(\band)}
   & \leq \sum_{j=0}^m \binom{m}{j} \frac{\rho^m}{m!} \| ((\partial_x^j a) \cdot \nabla) (\partial_x^{m-j} b) \|_{H^1(\band)}
   \\
   & \leq \sum_{j=0}^m \binom{m}{j} \frac{\rho^m}{m!}
   \frac{j!}{\rho_b^j} \grhob_{1,2}(a) \frac{(m-j)!}{\rho_b^{m-j}} \grhob_{2,2}(b)
   \\
   & \leq (m+1) \left(\frac{\rho}{\rho_b}\right)^m \grhob_{1,2}(a) \grhob_{2,2}(b)
   \\
   & \leq C\grhob_{1,2}(a) \grhob_{2,2}(b),
  \end{split}
 \end{equation}
 for some $C > 0$ independent of $m$ because $\rho < \rho_b$.  Hence,
 \begin{equation}
  \sum_{k\in\Z} \sup_{t\geq0} \| \oprhozero \D_k (u^1(t)\cdot\nabla)u^1(t) \|_{L^2(\band)} < +\infty.
 \end{equation}

 \item \emph{Third and fourth terms}. We apply \eqref{type1} to $g := \e W^\e \partial_x u^1 + \e (u^1\cdot\ey)\p_y w^\e$. Since $W^\e$ does not depend on $x$, we use the estimate
 \begin{equation}
  \begin{split}
  \| \oprhozero \D_k g(t) \|_{L^2(\band)}
  & \leq \e (\|  W^\e \|_\infty + \| \partial_y w^\e \|_\infty) \times
  \\
  & (\| \oprhozero \D_k \partial_x u^1(t) \|_{L^2(\band)} + 
  \| \oprhozero \D_k u^1(t) \|_{L^2(\band)})
  \\
  & \leq \e^{\frac14} C_W (\| \oprhozero \D_k \partial_x u^1(t) \|_{L^2(\band)} + 
  \| \oprhozero \D_k u^1(t) \|_{L^2(\band)})
  \end{split}
 \end{equation}
 thanks to estimate \eqref{estimate.W} from \cref{lemma.W}. We proceed as above using \eqref{u1io} to get uniform bounds on the sums. 
 
 \item \emph{Fifth term}. We apply \eqref{type2} to $g := \chi \p_x u^1$ and $\mathcal{V} := V$. The integral in time of the boundary layer is finite thanks to \eqref{ineq.V.decay} from \cref{Lemma:V} because $n\geq3$. The sum in $k$ is also finite using the same techniques as above.
 
 \item \emph{Sixth term}. We apply \eqref{type2} to $g := \chi' (u^1 \cdot \ey) / \varphi$ and $\mathcal{V} := z V$. The integral in time of the boundary layer is finite thanks to \eqref{ineq.V.decay} from \cref{Lemma:V} because $n\geq3$. The sum in $k$ is also finite using the same techniques as above.

 \item \emph{Seventh term}.  We apply \eqref{type2} to $g := \chi \varphi' (u^1 \cdot \ey) / \varphi$ and $\mathcal{V} := z \p_z V$. The integral in time of the boundary layer is finite thanks to \eqref{ineq.V.decay} from \cref{Lemma:V} because $n\geq3$. The sum in $k$ is also finite using the same techniques as above.

\end{itemize}

\section*{Acknowledgements}

The authors warmly thank Jean-Pierre Puel for his advice concerning the local controllability of the Navier-Stokes equations, and Marius Paicu for his insight on analytic regularization properties for Navier-Stokes systems. The authors also thank an anonymous referee for suggesting \cref{rk:no_carleman}.

\bigskip

\mbox{J.-M.\,Coron} and F.\,Marbach were supported by ERC Advanced Grant 266907 (CPDENL) of the 7th Research Framework Programme (FP7). 
\mbox{J.-M.\,Coron} also benefits from the supports of ANR project Finite4SoS (ANR 15 CE23 0007) and  Laboratoire sino-français en mathématiques appliquées (LIASFMA). 
F.\,Sueur was supported by the Agence Nationale de la Recherche, Project DYFICOLTI, grant ANR-13-BS01-0003-01, Project IFSMACS, grant ANR-15-CE40-0010 and Project  BORDS, grant ANR-16-CE40-0027-01, the Conseil Régional d’Aquitaine, grant 2015.1047.CP, and by the H2020-MSCA-ITN-2017 program, Project ConFlex, Grant ETN-765579. 
P.\,Zhang is partially supported by NSF of China under Grant 11371347.
F.\,Sueur warmly thanks Beijing's Morningside center for Mathematics for its kind hospitality a stay in September 2017.
F.\,Marbach and F.\,Sueur are grateful to ETH-ITS and ETH Zürich for their kind hospitality during two stays respectively in October and November 2017.

\bibliographystyle{plain}
\bibliography{bibliography}

\end{document}